%% file: OptimalCobsbetweenTorusKnots.tex
\documentclass[]{amsart}
\usepackage{amssymb, amsthm, amsmath, pifont} %f{\"u}r Symbole der reellen Zahlen
\usepackage[usenames]{color}
\usepackage{amsfonts}
\usepackage{enumerate}
\usepackage{pdfpages}
\usepackage[all]{xypic}
\usepackage{graphicx}
\usepackage{psfrag}
\usepackage{verbatim}
\usepackage{pstricks}
\usepackage[latin1]{inputenc}
\usepackage{mathrsfs}
\usepackage[all,knot]{xy}
\newtheorem{theorem}{Theorem}%[chapter]

\newtheorem{corollary}[theorem]{Corollary}

\newtheorem{lemma}[theorem]{Lemma}
\newtheorem{question}[theorem]{Question}
\newtheorem{example}[theorem]{Example}
\newtheorem{definition}[theorem]{Definition}
\newtheorem*{intdef}{Definition}
\newtheorem{remark}[theorem]{Remark}
\newtheorem{prop}[theorem]{Proposition}

\newtheorem{claim}[theorem]{Claim}
\newtheorem{obs}[theorem]{Observation}

\newenvironment{Example}{\begin{example}\rm}{\end{example}}
\newenvironment{Definition}{\begin{definition}\rm}{\end{definition}}

\newenvironment{Remark}{\begin{remark}\rm}{\end{remark}}

\newenvironment{Obs}{\begin{obs}\rm}{\end{obs}}
\newenvironment{Question}{\begin{question}\rm}{\end{question}}
\def\et{\;\mbox{and}\;}
\def\vel{\;\mbox{or}\;}

\def\s{{\sigma}}
\def\so{{\sigma_\omega}}

\def\l{{{l}}}

\def\et{\quad\mbox{and}\quad}
\def\vel{\quad\mbox{or}\quad}

\def\epsilon{\varepsilon}

\def\R{\mathbb{R}}

\def\C{\mathbb{C}}

\def\Re{\rm{Re}}
\begin{document}
\address{Boston College, Department of Mathematics, Maloney Hall, Chestnut Hill, MA 02467, United States}
\email{peter.feller.2@bc.edu}
\title{Optimal Cobordisms between Torus Knots}
\author{Peter Feller}
\thanks{The author gratefully acknowledges support by the Swiss National Science Foundation Grant 155477.}
\begin{abstract}
We construct cobordisms of small genus between torus knots and use them to
determine the cobordism distance between torus knots of small braid index.
In fact, the cobordisms we construct arise
as the intersection of a smooth algebraic curve in $\C^2$ with the
unit
4-ball %$B^4$
from which a $4$-ball of smaller radius is removed.
Connections to the realization problem of $A_n$-singularities on algebraic plane curves
and the adjacency problem for plane curve singularities are discussed.
To obstruct the existence of cobordisms, we use Ozsv\'ath, Stipsicz, and Szab\'o's
$\Upsilon$-invariant, which we provide explicitly for torus knots of braid index 3 and~4.
\end{abstract}
\maketitle
\section{Introduction}
For a %(smooth, oriented)
knot $K$\textemdash a smooth and oriented embedding of the unit circle $S^1$ into the unit $3$-sphere $S^3$\textemdash the \emph{slice genus} $g_s(K)$ is the minimal genus of %properly
smooth, oriented
surfaces $F$ in the closed unit $4$-ball $B^4$ with oriented boundary $%\partial F=
K\subset \partial B^4=S^3$.
For torus knots, the slice genus is %known to be
equal to their genus $g$; i.e., for coprime positive integers $p$ and $q$, one has
 \begin{equation}\label{eq:TCforTK}
 g_s(T_{p,q})=g(T_{p,q})=\frac{(p-1)(q-1)}{2}%,\quad\text{for coprime positive integers }p\text{ and }q
 .\end{equation}
 More generally, the slice genus is known for knots $K$ that arise as the transversal intersection of $S^3$ with a smooth algebraic curve $V_f$ in $\C^2$,
 by Kronheimer and Mrowka's resolution of the Thom conjecture~\cite[Corollary 1.3]{KronheimerMrowka_Gaugetheoryforemb}:
the surface in $B^4\subset\C^2$ given as the intersection of $B^4$ with $V_f$ has genus $g_s(K)$; see also Rudolph's slice-Bennequin inequality~\cite{rudolph_QPasObstruction}.
%The slice genus %of knots is difficult to calculate; for example, it
%is in general unknown for connect sums of torus knots.
We determine the slice genus for connected sums of torus knots of small braid index; see Corollary~\ref{cor:cobdist}.

The \emph{cobordism distance} %\footnote{The name ``distance'' is justified by the fact that $d_c$ descends to a metric on the smooth concordance group}
$d_c(K,T)$ between two knots $K$ and $T$
is the minimal genus of %properly
smoothly embedded and oriented %oriented smooth
surfaces $C$ in $S^3\times[0,1]$
with boundary $K\times\{0\}\cup T\times\{1\}$ such that the induced orientation agrees with the orientation of $T$ and disagrees with the orientation of $K$.
Such a $C$ is called a \emph{cobordism} between $K$ and $T$.
Equivalently,
$d_c(K,T)$ can be defined as the slice genus of the connected sum $K\sharp-{T}$ of $K$ and $-{T}$\textemdash the mirror image of $T$ with reversed orientation.
Cobordism distance satisfies the triangle inequality;
in particular, $d_c(K,T)\geq |g_s(T)-g_s(K)|$ for all knots $K$ and $T$.
We call a cobordism between two knots \emph{algebraic} if it arises as the intersection of a smooth algebraic curve in $\C^2$ with
$\overline{B_2^4\backslash B_1^4}\cong S^3\times[0,1]$, where $B_i^4\subset\C^2$ are the $4$-balls centered at the origin of radius $r_i$ for some $0<r_1<r_2$. %\in\R$.
By the Thom conjecture, algebraic cobordisms between two knots $K$ and $T$ have genus $|g_s(T)-g_s(K)|$. In particular, the existence of an algebraic cobordism between $K$ and $T$ does determine
their cobordism distance to be $|g_s(T)-g_s(K)|$. We call a cobordism $C$ between two knots $K$ and $T$ \emph{optimal}
if its genus $g(C)$ is $|g_s(T)-g_s(K)|$.

We
address the existence of algebraic and optimal cobordisms for torus knots.

\begin{theorem}\label{thm:32}
For positive torus knots $T_{2,n}$ and $T_{3,m}$ of braid index 2 and 3, respectively, the following are equivalent.
\begin{enumerate}[(I)]
\item\label{1}
There exists an optimal cobordism between $T_{2,n}$ and $T_{3,m}$; i.e.\ %it is
\[d_c(T_{2,n},T_{3,m})=
g_s(T_{2,n}\sharp -{T_{3,m}})=\left|g_s(T_{3,m})-g_s(T_{2,n})\right|=\left|m-1-\frac{n-1}{2}\right|.\]
\item \label{2}$n\leq\frac{5m-1}{3}.$
\item \label{3}There exists an algebraic cobordism between $T_{2,n}$ and $T_{3,m}$.
\end{enumerate}
\end{theorem}
\begin{theorem}\label{thm:42}
For positive torus knots $T_{2,n}$ and $T_{4,m}$ of braid index 2 and 4, respectively, the following are equivalent.
\begin{enumerate}[(I)]
\item
There exists an optimal cobordism between $T_{2,n}$ and $T_{4,m}$; i.e.\ %it is
\[d_c(T_{2,n},T_{4,m})=
g_s(T_{2,n}\sharp -{T_{4,m}})=\left|g_s(T_{4,m})-g_s(T_{2,n})\right|=
\left|\frac{3m-n-2}{2}\right|.
\]
\item %$g(T_1)\leq \frac{5m-5}{4}.$
$n\leq\frac{5m-3}{2}.$
\item There exists an algebraic cobordism between $T_{2,n}$ and $T_{4,m}$.
\end{enumerate}
\end{theorem}

Theorem~\ref{thm:32} and Theorem~\ref{thm:42} are established using the following strategy.
If~\eqref{2} holds, we provide an explicit construction of the optimal cobordism using positive braids. In fact, the optimal cobordisms we find, can be %explicitly
seen in $S^3$ as a sequence of positive destabilizations on the fiber surfaces of the larger-genus knot; see Remark~\ref{Rem:PosDestab}. If~\eqref{2} does not hold, we use the %new
$\upsilon$-invariant to obstruct the existence of an optimal cobordism. Here, $\upsilon=\Upsilon(1)$ is one of a family $\Upsilon(t)$ of concordance invariants introduced by Ozsv\'ath, Stipsicz, and Szab\'o~\cite{OSS_2014}, which generalize Ozsv\'ath and Szab\'o's $\tau$-invariant as introduced in~\cite{OzsvathSzabo_03_KFHandthefourballgenus}.
Finally, all optimal cobordisms we construct turn out to be algebraic.
We establish a more general result for all knots that arise as the transversal intersection of $S^3$ with a smooth algebraic curve in $\C^2$: the natural way of constructing optimal cobordisms always yields algebraic cobordisms;  see Lemma~\ref{lemma:algrealization}.
This brings us to ask: if there exists an optimal cobordism between two positive torus knots, does there exist an algebraic cobordism between them?
The proof of Lemma~\ref{lemma:algrealization} uses realization results of Orevkov and Rudolph.
Using deplumbing, we also construct algebraic
cobordisms
between $T_{2,n}$ and $T_{m,m+1}$; see Section~\ref{subsec:ddto2n}.
This is related to work of Orevkov~\cite{Orevkov_12_SomeExamples}, see Remark~\ref{rem:TmmtoT2n}, and motivated by algebraic geometry questions; see Section~\ref{sec:motivation}.

We now turn to the cobordism distance between torus knots.
By gluing together the optimal cobordisms given in Theorems~\ref{thm:32} and~\ref{thm:42}, we will obtain the following.
\begin{corollary}\label{cor:cobdist}
Let $K$ and $T$ be torus knots such that the sum of their braid indices is 6 or less. Then we have the following formula for their cobordism distance.
\[d_c(T,K)=g_s(K\sharp-{T})=\max\{|\tau(K)-\tau(T)|,|\upsilon(K)-\upsilon(T)|\}.\]
\end{corollary}

The values of the concordance invariants
$\tau$ and $\upsilon$ that arise in Corollary~\ref{cor:cobdist} %=\Upsilon_K(1)$
are explicitly calculable:
For positive torus knots (for their negative counterparts, which are obtained by taking the mirror image, the sign changes), one has
\[\tau(T_{p,q})=-g(T_{p,q})=-\frac{(p-1)(q-1)}{2}\text{~\cite[Corollary~1.7]{OzsvathSzabo_03_KFHandthefourballgenus}},\]
for coprime positive integers $p$ and $q$, and \begin{equation}\label{upsilonsfortk}
\upsilon(T_{2,2k+1})%=\tau(T_{2,2n+1})
=-k,\quad
\upsilon(T_{3,3k+1})=\upsilon(T_{4,2k+1})=-2k,\quad
\upsilon(T_{3,3k+2})=-2k-1,
\end{equation}
for positive integers $k$.
The $\upsilon$-values for torus knots of braid index 2 follow, for example, from the fact that
$\upsilon$ equals half the signature for alternating knots~\cite[Theorem 1.14]{OSS_2014}. The other $\upsilon$-values in~\eqref{upsilonsfortk} can be derived from the inductive formula for $\upsilon$ provided by Ozsv\'ath, Stipsticz and Szab\'o~\cite[Theorem 1.15]{OSS_2014}. We do this in Section~\ref{sec:upsilonfortorusknots}.

Of course, part of the statement of Corollary~\ref{cor:cobdist} was known before; for example in the cases covered by the following remark.
\begin{Remark}\label{Rem:cor:cobdist} %Corollary~\ref{cor:cobdist} subsumes the following well-established facts.
If $K$ and $T$ are positive and negative torus knots, respectively, one has
\[d_c(T,K)=g_s(K\sharp -{T})=g_s(K)+g_s(T)=g(K)+g(T)\] by the Thom conjecture.
If $K$ and $T$ are both positive (both negative) torus knots of the same braid index, then there exists an optimal cobordism %of genus $|g(T)-g(K)|$
between them; compare Example~\ref{Ex:incompTnm<Tabifn<am<b}.
\end{Remark}
Determining the cobordism distance between all torus knots seems out of reach.
However, a coarse estimation of the cobordism distance between torus knots was provided by Baader~\cite{Baader_ScissorEq}.

The study of optimal and algebraic cobordisms between torus knots seems natural from a knot theoretic point of view.
We discuss additional motivation from algebraic geometry questions in Section~\ref{sec:motivation}.
Section~\ref{sec:notations} recalls the notions of positive and quasi-positive braids. In Section~\ref{sec:algrealization} we show that optimal cobordisms given by quasi-positive braid sequences are algebraic.
In Section~\ref{sec:constructionofcobs} %all
optimal cobordisms
are constructed
and used to prove Theorem~\ref{thm:32}, Theorem~\ref{thm:42}, and Corollary~\ref{cor:cobdist}.
Section~\ref{sec:upsilonfortorusknots} provides the $\Upsilon$-values for torus knots of braid index $3$ and $4$.

{\bf{Acknowledgements}:} I thank Sebastian Baader for sharing his insight into cobordisms
and for his ongoing support.
Thanks also to Adam Levine for pointing me towards the $\Upsilon$-invariant, and Immanuel Stampfli and Andrew Yarmola for valuable suggestions. Finally, I wish to thank the referee
for helpful suggestions and corrections.

\section{Algebraic motivation: Plane curve singularities over $\C$}\label{sec:motivation}
In this section, we discuss motivations to study algebraic and optimal cobordisms between torus knots coming from singularity theory.
Mathematically, the rest of the paper is independent of this.

We consider isolated singularities on algebraic curves in $\C^2$ and we denote singularities by the function germs %in $\C\{x,y\}$
that define them.
A general question asks, %Fixing a positive integer $d$ it is a general question,
what (topological) type of singularities can occur on an algebraic curve $V_f$ in $\C^2$ given as the zero-set of a square-free polynomial $f$ in $\C[x,y]$ of some fixed degree $d$; see e.g.~Greuel, Lossen, and Shustin's work~\cite{GreuelLossenShustin_89_PlaneCurveswithPrescribedSing}. Even for simple singularities\textemdash those %with positive intersection form thus
corresponding to Dynkin diagrams~\cite{Arnold_normalforms}\textemdash a lot is unknown. For the %simplest family among the simple
$A_k$-singularities\textemdash the simple singularities given by $y^2-x^{k+1}$\textemdash the following bounds were provided by Guse\u{\i}n-Zade and Nekhoroshev
 \begin{equation}\label{eq:asymbounds}
 %\frac{1}{2}d^2+O(d)<
 \frac{15}{28}d^2+O(d)\leq%\limsup_{d\to\infty}
 k(d)\leq\frac{3}{4}d^2+O(d)\;\;\text{\cite{Gusein-ZadeNekhoroshev_00_OnSingofTypeAk}},
 \end{equation}
where $k(d)$ denotes the maximal integer $k$ such that $A_k$ occurs on a degree $d$ curve.
In fact, Orevkov improved the lower bound to
\begin{equation}\label{eq:orevkov}\frac{7}{12}d^2+O(d)\leq k(d)\quad\text{\cite{Orevkov_12_SomeExamples}}.\end{equation}

Recall that, for a singularity at $p=(p_1,p_2)\in V_f$, its \emph{link of singularity}
is the link obtained as the transversal intersection of $V_f$ with the small $3$-sphere
\[S^3_\varepsilon=\{(x,y)\;|\;|x-p_1|^2+|y-p_2|^2=\varepsilon^2\}\subset\C^2,\] for small enough $\varepsilon>0$;
see %e.g.~
Milnor~\cite{milnor_SingularPoints}.
Similarly, the \emph{link at infinity} of an algebraic curve $V_f$ is defined to be the transversal intersection of $V_f$ with the $3$-sphere $S^3_R\subset\C^2$ of radius $R$ for $R$ large enough; see e.g.~Neumann and Rudolph~\cite{Neumann_Rudolph_87_Unfoldings}.

Prototypical examples of plane curve singularities are the singularities %at the origin of the zero-set of
$f_{p,q}=y^p-x^q$, where $p$ and $q$ are positive integers. They have the torus link $T_{p,q}$ as link of singularity. Up to topological type, singularities are determined by their link of singularity and the links that arise as links of singularities are fully understood: their components are positive torus knots or special cables thereof; see e.g.~Brieskorn and Kn\"orrer's book~\cite{brieskornknoerrer}.

In the algebraic setting it is natural to consider links not only knots.
Algebraic cobordisms between links are defined as for knots. Optimal cobordisms are defined via Euler characteristic instead of the genus; see Section~\ref{sec:algrealization}.

\begin{Obs}\label{Obs:algcob}
The existence of a singularity on a curve of degree $d$ implies the existence of an optimal cobordism from the link of the singularity to $T_{d,d}$%(and to $T_{d,d+1}$)
. In particular, there exists an optimal cobordism from $T_{2,k(d)+1}$ to $T_{d,d}$. % and $T_{d,d+1}$.
\end{Obs}
Observation~\ref{Obs:algcob} motivated our study of optimal cobordisms from $T_{2,n}$ to $T_{d,d}$ and $T_{d,d+1}$.
In Section~\ref{subsec:ddto2n} we show that there exist algebraic cobordisms between $T_{2,n}$ and $T_{d,d}$ ($T_{d,d+1}$) if $n\leq \frac{2}{3}d^2+O(d)$. In particular, %Propositions~\ref{prop:subsurfacesofTmm} and~\ref{prop:subsurfacesofTm(m+1)} show that
no obvious topological obstruction exists to having $k(d)\geq\frac{2}{3} d^2+O(d)$; compare also~\cite{Orevkov_12_SomeExamples}. %We note that
Observation~\ref{Obs:algcob} allows to give a knot theoretic  proof of the upper bound in~\eqref{eq:asymbounds}; see Remark~\ref{Rem:subwordsofTmm}.
To establish Observation~\ref{Obs:algcob}, we note that, whenever a singularity occurs on an algebraic curve $V_f$, we get an algebraic cobordism from the link of the singularity $K$
to the link at infinity $T$ of $V_f$.
For this, %choose $\varepsilon$ small enough and $R$ large enough such that $S^3_\epsilon$ and $S^3_R$ intersect $V_g$ transversally in $K$ and $T$, respectively.
choose a small sphere $S^3_\epsilon$ and a large sphere $S^3_R$ that intersect $V_f$ transversally in $K$ and $T$, respectively.
Let $V_g$ be another algebraic curve. By transversality,
$V_f$ and $V_g$ intersects $S^3_\varepsilon$ and $S^3_R$ in the same links as long as $g$ and $f$ are ``close''. To be precise, this is certainly true if $g=f+t$ and $t\in\C$ is small. % and non-zero.
For generic $t$, $V_g$ is smooth; thus, $V_g$ yields an algebraic cobordism between $K$ and $L$.
Furthermore, %if $L$ is the link at infinity of an algebraic curve $V_f$ with $f$ of degree $d$, then
there is an algebraic cobordism from $T$ to $T_{d,d}$. This follows by using that the link at infinity of $f+s(x^d+y^d)$ is $T_{d,d}$, for generic $s\in\C$, while $S^3_R\pitchfork V_{f+s(x^d+y^d)}$ is $T$ for $s$ small enough; and then arguing as above. %; thus, one can argue as above choosing $g$ ``close'' to $f+s(x^d+y^d)$. %This is seen as follows. Assume (by a linear change of coordinates) that $f=y^d+c_{d-1}(x)y^{d-1}+\cdots+c_0(x)$ and fix $R>0$ such that $T$ is $S^3_R\pitchfork V_f$.
%The link at infinity of $f+sx^d$ %($f+sy^{d+1}$)
%is $T_{d,d}$ %($T_{d,d+1}$)
%for generic $s\in\C$, which follows for example from a variant of a result by Kouchnirenko~\cite[Corollaire 1.22]{Kouchnirenko_NewtonPoly&MilnorNr}.
%Therefore,
%we get an algebraic cobordism from $T$ to $T_{d,d}$ %($T_{d,d+1}$
%by choosing a small generic $s$ and using the smooth algebraic curve $V_g$ given by some $g$ ``close'' to $f+sx^d$ as above.
Gluing the two algebraic cobordisms together yields an optimal cobordism from $K$ to $T_{d,d}$.
%A similar argument works for $T_{d,d+1}$.

A related question asks about the existence of adjacencies between singularities.
Fixing a singularity $f$,
another singularity $g$ is said to be \emph{adjacent} to $f$
if there exists a smooth family of germs $f_t$ such that $f_0\cong f$ and $g\cong f_t$ for small enough non-zero $t$. There are different notions of equivalence $\cong$ yielding different notions of adjacencies.
However, as long as $g$ defines a simple singularity the notions all agree. See Siersma's work for a discussion of these notions~\cite{siersma_diss} and compare also with Arnol$'$d's work, who was the first to fully describe adjacency between simple singularities~\cite[Corollary~8.7]{Arnold_normalforms}. A modern introduction to singularities and their deformations is provided by Greuel, Lossen and Shustin~\cite{GreuelLossenShustin_Intro}.

If $g$ is adjacent to $f$, then there exists an algebraic cobordism between their links of the singularity (given by $V_{f_t+\varepsilon}$ for $t$ and $\varepsilon$ small as a similar argument as above shows). The adjacency question is mostly unresolved if $f$ is not a simple singularity. A natural first case to consider is $f=f_{p,q}$ for fixed $p>2$ and to ask: Given a positive integer $q$, which $A_{n}$-singularities are adjacent to $f_{p,q}$? Theorem~\ref{thm:32} and Theorem~\ref{thm:42} can be seen as answering analogs of this question for $p$ equal to $3$ and $4$, respectively. %This led to ask about optimal cobordism between torus knots. %Theorems~\ref{thm:32} and~\ref{thm:42} can be see as

\section{Braids and (quasi-)positivity}\label{sec:notations}
To set notions, we shortly recall Artin's braid group~\cite{Artin_TheorieDerZoepfe}; a nice reference for braids is Birman's book~\cite{Birman_74_BraidsLinksAndMCGs}.
Let us fixe a positive integer $n$. The standard group presentation for the \emph{braid group on $n$ strands}, denoted by $B_n$, is given by generators $a_1,\cdots, a_{n-1}$ subject to the \emph{braid relations}
\[a_ia_{i+1}a_i=a_{i+1}a_{i}a_{i+1}\text{ for }\; 1\leq i\leq n-2\et
 a_ia_j=a_ja_i \text{ for }\; |i-j|\geq 2.\]

 Elements $\beta$ of $B_n$, called \emph{braids} or \emph{n-braids}, have a well-defined (algebraic) length $\l(\beta)$, given by the number of generators minus the number of inverses of generators in a word representing $\beta$.
 More geometrically, a $n$-braid $\beta$ can be viewed as an isotopy class of an oriented compact $1$-submanifold of $[0,1]\times \C$ such that the projection to $[0,1]$ is a $n$-fold orientation-preserving regular map and $\beta$ intersects $\{0\}\times\C$ and $\{1\}\times\C$ in $\{0\}\times P$ and $\{1\}\times P$, respectively, where $P$ is a subset of $\C$ consisting of $n$ complex numbers with pairwise different real part. The above standard generators $a_i$ are identified with the braid that exchanges the $i$th and $i+1$th (with respect to order induced by the real order) point of $P$ by a %clockwise
 half-twist parameterized by $[0,1]$ and the group operations is given by stacking braids on top of each other.
 The \emph{closure} $\overline{\beta}$ of $\beta$ is the closed $1$-submanifold in $S^1\times\C$ obtained by gluing the top of $\beta\subset [0,1]\times\C$ to its bottom. A closed braid $\overline{\beta}$ yields a link in $S^3$, also denoted by $\overline{\beta}$, via a fixed standard embedding of the solid torus $S^1\times\C$ in $S^3$. The \emph{braid index} of a link is the minimal number $n$ such that $L$ arises as the closure of a $n$-braid.

\emph{Positive braids} are the elements of the semi-subgroup $B_{n,+}$ that is generated by all the generators $a_i$.
\emph{Positive torus links} are examples of links that arise as closures of positive braids:
For positive integers $p$ and $q$, the closure of $(a_1a_2\cdots a_{p-1})^q$ is $T_{p,q}$, which is a knot, called a \emph{positive torus knot}, if and only if $p$ and $q$ are coprime. The braid index of $T_{p,q}$ is $\min\{p,q\}$.

Rudolph introduced \emph{quasi-positive braids}\textemdash the elements of the semi-subgroup of $B_n$ generated by all conjugates of the generators $a_i$; i.e.~the braids given by \emph{quasi-positive braid words} $\prod_{k=1}^l \omega_ka_{i_k}\omega_k^{-1}$; compare~\cite{Rudolph_83_AlgFunctionsAndClosedBraids}.
A knot or link is called \emph{quasi-positive} if it arises as the closure of a quasi-positive braid. A quasi-positive braid $\beta$ has an associated canonical ribbon surface $F_\beta$ embedded in $B^4$ with the closure of $\beta$ as boundary, which can be seen in $S^3$ given by $n$ discs, one for every strand, and $\l(\beta)$ ribbon bands between the discs. In particular, the Euler characteristic $\chi$ of $F_\beta$ is $n-\l(\beta)$.
By the slice-Bennequin inequality~\cite{rudolph_QPasObstruction}, $\chi(F_{\beta})$ equals $\chi_s(\overline{\beta})$\textemdash the maximal Euler characteristic among all %properly
oriented and smoothly
embedded surfaces $F$ (without closed components) in $B^4$ such that $\partial F\subset S^3$ is the link $\overline{\beta}$.

 Rudolph established that all quasi-positive links %$\overline{\beta}$
 arise as the transversal intersection of a smooth algebraic curve in $\C^2$ with $S^3$%(and the intersection with $B^4$ is ribbon-isotopic to $F_\beta$)
 ~\cite{Rudolph_83_AlgFunctionsAndClosedBraids}.
 Boileau and Orevkov proved that this is a characterization of quasi-positive links~\cite[Theorem~1]{BoileauOrevkov_QuasiPositivite}.

\section{Algebraic realization of optimal cobordisms}
\label{sec:algrealization}
This section is concerned with establishing the following realization Lemma.
\begin{lemma}\label{lemma:algrealization}
Let $\beta_1$ and $\beta_2$ be quasi-positive $n$-braid words. % representing braids on $n$ strands.
If $\beta_2$ can be obtained from $\beta_1$ by applying a finite number of braid group relations,
conjugations,
and additions of a conjugate of a %positive
generator anywhere in the braid word;
then there exists an algebraic cobordism $C$ between the links obtained as the closures of the $\beta_i$.
In fact, $C$ is given as the zero-set of a polynomial in $\C[x,y]$ of the form
\[y^n+ c_{n-1}(x)y^{n-1}+\cdots+c_0(x),\]
where the $c_i(x)$ are polynomials.
\end{lemma}
Let $\beta_1$ and $\beta_2$ be quasi-positive braid words given as in Lemma~\ref{lemma:algrealization}; i.e.~there is a sequence of $n$-braid words $(\alpha_1,\cdots,\alpha_k)$ starting with $\beta_1$ ending with $\beta_2$ such that $\alpha_j$ and $\alpha_{j+1}$ either define the same conjugacy class in $B_n$ or $\alpha_{j+1}$ is obtained by adding a generator $a_i$ somewhere in $\alpha_j$.
There is an associated cobordism $C$ between $\overline{\beta_1}$ and $\overline{\beta_2}$ given (as a handle decomposition) by 1-handle attachments
corresponding to every %positive
generator that is added. The cobordism $C$ is \emph{optimal}; i.e.~it has Euler characteristic
\[\chi_s(\overline{\beta_2})-\chi_s(\overline{\beta_1})=\l(\beta_1)-\l(\beta_2)\]
 and does not have closed components (this is the sensible extension of the notion of optimal cobordisms to links).
 In fact, although not made explicit, the proof of Lemma~\ref{lemma:algrealization} given below does show that this $C$ is algebraic.
All optimal cobordisms we construct in Section~\ref{sec:constructionofcobs} arise %as a sequence of braid words
as described above. %in Lemma~\ref{lemma:algrealization}.
We see Lemma~\ref{lemma:algrealization} as evidence that the following question might have a positive answer.
\begin{Question}\label{q:algrealization}
%Let $K$ and $T$ be quasi-positive knots such that $d_c(K,T)=g_s(T)-g_s(K)$. Do there exist braids $\beta_1$ and $\beta_2$ as described in Lemma~\ref{lemma:algrealization} with closure $K$ and $T$, respectively?
Are the %following
two necessary conditions for the existence of an algebraic cobordism between two knots\textemdash the knots are quasi-positive and there exists an optimal cobordism between them\textemdash sufficient?
\end{Question}
The proof of Lemma~\ref{lemma:algrealization} occupies the rest of this section and uses Rudolph diagrams. Only the statement of Lemma~\ref{lemma:algrealization} is used in the rest of the paper.
\subsection{Rudolph diagrams}\label{subsec:RD}
To set notation and for the reader's convenience, we recall the notion of Rudolph diagrams, following~\cite{Rudolph_83_AlgFunctionsAndClosedBraids} and~\cite{Orevkov_96_Rudolph}.

For a square-free algebraic function $f\colon \C^2\to\C$ of the form
\[y^n+ c_{n-1}(x)y^{n-1}+\cdots+c_0(x) \in\C[x,y],\]
we study the following subsets of $\C$.

\textbullet\quad The finite set {$B$} of all $x$ such that some of the $n$ solutions $y_1,\ldots, y_n$ of $f(x,y)=0$ coincide.

\textbullet\quad The semi real-analytic set {$B^+$} of all $x$ such that
the $n$ solutions of $f(x,y)=0$ are all different, but do not have $n$ distinct real parts.

Their union $B\cup B^+$ is denoted by $G(f)$.
\begin{Example}\label{Ex:G(f)} If $f$ is $y^2+x$, then $B=\{0\}$, $B^+=(0,\infty)$, and $G(f)=[0,\infty)$.
\end{Example}

Let $V_f\subset\C$ denote the zero-set of $f$. For an oriented simple closed curve $\gamma$ in $\C\backslash B$, the intersection $V_f\cap(\gamma\times \C)\subset \gamma\times \C$ is a closed $n$-braid via the identification $\gamma\times \C\cong S^1\times \C$.
Similarly,
for every oriented arc $\alpha$ in $\C\setminus B$ with endpoints in $\C\setminus G(f)$ (which guarantees that at endpoints the $n$-solutions have different real parts),
the intersection $V_f\cap(\alpha\times \C)\subset\alpha\times \C$ is a $n$-braid by identifying $\alpha\times \C$ with $[0,1]\times \C$.
Note that for this to be well-defined, the identification should preserve the order of the real parts in the second factor.
An endpoint-fixing isotopy of two arcs and an isotopy of two simple closed curves in $\C \setminus B$ correspond to an isotopy of braids and an isotopy of closed braids, respectively. Any choice of convention not made explicit so far is chosen such that in Example~\ref{Ex:G(f)} the oriented arc starting at $1-i$ and ending at $1+i$ yields the $2$-braid~$a_1$.

Let $\pi\colon\C^2\to\C$ be the projection to the first coordinate.
We will only consider $f$ such that $f=0$ defines a non-singular algebraic curve $V_f\subset\C^2$ and such that for every $x$ in $B$ the intersection $\pi^{-1}(x)\cap V_f$ consists of exactly $n-1$ points; i.e.~fixing an $x$ in $B$ gives a polynomial in $y$ with precisely one
repeated root of multiplicity two. % whos $y$-coordinate all have different real parts.
Rudolph observed that $G(f)=B\cup B^+$ naturally carries the structure of an oriented, $\{1,\cdots,n-1\}$-labeled graph that describes $V_f$ up to $\pi$-preserving smooth isotopy in $\C^2$.
 After a generic small linear coordinate change (to rule out pathologies), the vertices are locally given as in Figure~\ref{fig:vertices}.
 \begin{figure}[h]
\centering
\def\svgscale{1.5}
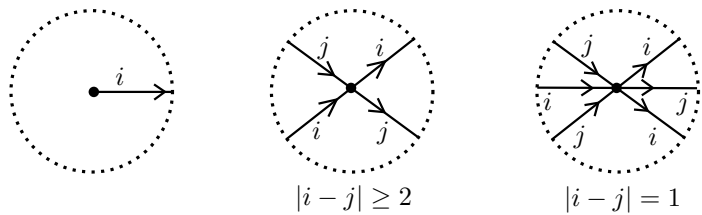
\caption{Neighborhoods of vertices of $G(f)$.}
\label{fig:vertices}
\end{figure}
The elements of $B$ are the $1$-valent vertices.
The edges are the connected components of the semi real-analytic open subset of $B^+$ given by those $x$ that %are not in $B$ and
have precisely $n-1$ different real parts among the real parts of the solutions $y_1,\ldots,y_n$ to $f(x,y)=0$. Some of the edges tend to infinity instead of ending at a vertex. An edge $e$ is labeled as follows.
 For $x$ in $e\subset\C$, the $n$ solutions $y_i$ of $f(x,y)$ can be indexed such that the index order agrees with the order given by their real parts; i.e.
 \[\Re(y_1)<\cdots<\Re(y_{k_e})=\Re(y_{k_e+1})<\cdots<\Re(y_n)\]
 for some $k_e$ in $\{1,\cdots,n-1\}$. The edge $e$ is labeled with $k_e$. The edges are oriented as follows.
 Pick a small oriented arc $\alpha\subset\C$ that meets $e$ transversally in a point $x$. The braid associated with $\alpha$ is either
 $a_{k_e}$ or $a_{k_e}^{-1}$ (the $k_e$th and $k_e+1$th solution exchange their real-part order while passing through $x$).
 Orient $e$ such that, if the orientation of $e$ followed by the orientation of $\gamma$ gives the
complex orientation of $\C$, then the braid corresponding to the transverse arc is $a_{e_k}$ (rather
than for $a_{e_k}^{-1}$).
In particular, %we arrange everything such that the
edges are oriented to point away from the $1$-valent vertices.

 The oriented, labeled graph $G(f)$ describes $V_f$ up to $\pi$-preserving isotopy; in particular, it describes all closed braids given by intersecting $V_f$ with cylinders. For a fixed embedded curve in $\gamma$ in $\C\setminus B$ with a marked start point $p\notin B^+$, one gets an explicit procedure, how to read off a braid word for the braid $\beta$ corresponding to the arc starting and ending at $p$ going counter-clockwise around $\gamma$:
by a small isotopy of $\gamma$ in $\C\setminus B$, we may assume that $\gamma$ meets $G(f)$ transversally and in edges only. Starting at $p$ we move counter-clockwise around $\gamma$.
Whenever we cross an edge $e$ transversally at a point $x$, %this edge has a label $i$ in $\{1,\cdots,n-1\}$.
we write down the generator $a_{k_e}$ or its inverse $a_{k_e}^{-1}$ depending on whether the orientation at $x$ given by $G(f)$ and $\gamma$ agrees or disagrees with the complex orientation of $\C$.

The study of the graphs $G(f)$ motivates the following %topological
definition.
Fix some surface $S$ with boundary. In fact, we will only consider cases where $S$ is either
\[\text{the unit disc }D=\left\{x\in\C\;\middle\vert\;|x|\leq 1\right\}\vel\text{the annulus }A=\left\{x\in\C\;\middle\vert\;1\leq|x|\leq2\right\}.\]
A  \emph{Rudolph diagram} on $S$
is an oriented, $\{1,\cdots,n-1\}$-labeled graph $G$ with smooth edges (we also allow smooth closed cycles) that enters and exits the boundary of $S$ transversely and is locally modelled on graphs $G(f)$ coming from an algebraic function $f\in\C[x,y]$ as above; i.e.~locally around a vertex $G$ is given as in Figure~\ref{fig:vertices}. We denote the set of $1$-valent vertices by $B\subset G$. Of course, a huge source of examples are obtained by embedding (or immersing) $S$ in $\C$ and using this embedding to define $G$ as the pull back of $G(f)$ for some algebraic function $f$.  Any closed curve $\gamma$ missing $1$-valent vertices defines a closed braid $\overline{\beta}$ by isotoping $\gamma$ to meet $G$ transversally in edges and then reading off a braid word $\beta$ for that closure as described in the $G(f)$-case.
A Rudolph diagram is said to be \emph{smooth} if it contains only $1$-valent vertices.
\begin{Example}\label{Ex:realizingQPasRD}
Given a quasi-positive braid word $\beta=\prod w_l a_{i_l} w_l^{-1}$, there exists a smooth Rudolph diagram on $D$ such that braid word read off when following $S^1=\partial D$ is $\beta$. Figure~\ref{fig:conjugaterealization} illustrates how one factor $\omega a_i \omega^{-1}$ is realized.
\begin{figure}[h]
\centering
\def\svgscale{1}
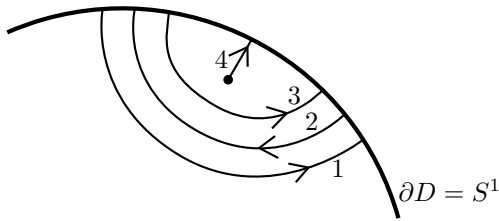
\caption{Piece of a Rudolph diagram that yields the braid word $\omega a_i\omega^{-1}$ when following the boundary, for $i=4$ and $\omega=a_1a_2^{-1}a_3$.}
\label{fig:conjugaterealization}
\end{figure}
\end{Example}
Orevkov describes which smooth Rudolph diagrams on $D$ arise as $G(f)$.
\begin{prop}\label{prop:RudolphRealization}
\cite[Proposition~2.1]{Orevkov_96_Rudolph}
Let $G$ be a smooth Rudolph diagram on $D$.
There exists an algebraic function
\[f=y^n+c_{n-1}(x)y^{n-1}+\cdots+c_0(x)\in\C[x,y]\]
such that $G$ is isotopic to $G(f)\cap D$ if and only if $G(f)
$ contains no closed cycles.% of an algebraic function on the unit disc.
\end{prop}
As pointed out by Orevkov, Rudolph (implicitly) proved such a statement while establishing the main theorem of~\cite{Rudolph_83_AlgFunctionsAndClosedBraids}.
%, when he established Remark~\ref{Rem:qusipos=braidsincylinders}.
%The main point in establishing Proposition~\ref{prop:RudolphRealization} is that every smooth Rudolph diagram with out closed cycles on $D$ can in fact be obtained from a prototypical polynomial $P$ by embedding $D$ into $\C$ and restricting $G(P)$ to this embedding.

\begin{Remark}\label{Rem:cycleremove}
Given a smooth Rudolph diagram $G$ on $D$ one can remove all closed cycles without changing the closed braids associated with closed curves in $D$.
\end{Remark}

\subsection{Rudolph diagrams on the annulus and braid word sequences}
For a Rudolph diagram $G$ on $A$, we denote by $\beta_1$ and $\beta_2$ the two braids defined by $G$ via reading off braid words following the inner and outer boundary of $A$ counter-clockwise starting at $1$ and $2$, respectively. For $\beta_1$ and $\beta_2$ to be well-defined, we ask that $G$ does not meet $1$ or $2$, which from now on is imposed on every Rudolph diagram.

For the proof of Lemma~\ref{lemma:algrealization}, we need the following. If a braid $\beta$ is obtained from a braid $\alpha$ as described in Lemma~\ref{lemma:algrealization}, then there exists a Rudolph diagram on $A$ such that $\alpha=\beta_1$ and $\beta=\beta_2$. The rest of this subsection provides one way of making this statement precise.

\begin{Remark}\label{Rem:genericoutwarddiagram}
Any Rudolph diagram $G$ on $D$ or $A$ can be isotoped such that it is \emph{outward-oriented}, which is defined as follows: All but a finite number of circles around the origin intersect $G$ transversally in edges. Furthermore, the finite exceptional circle meet $G$ transversally in edges except in one point $x$, which falls in one of two categories.
Either $x$ lies in the interior of an edge and the radial function restricted to that edge has a strict local extremum.
Or $x$ is a vertex and locally around $x$ the Rudolph Diagram $G$ and the exceptional circle behave as described in Figure~\ref{fig:verticesCC}.
\begin{figure}[h]
\centering
\def\svgscale{1.5}
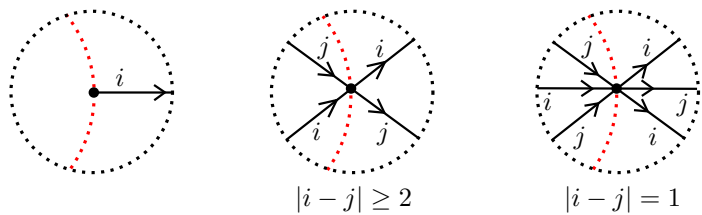
\caption{Neighborhoods of vertices of an outward-oriented Rudolph diagram (black) and how they meet their corresponding exceptional circles (red).}
\label{fig:verticesCC}
\end{figure}
%(up tand the exceptional circle
%Or $x$ is a vertex and %the edges starting or ending at $x$
%the radial function increases on $G$ locally around $x$ when following the orientation.
Finally, the positive real ray $[0,\infty)$ meets $G$ transversally in edges and away from the finite number of exceptional circles. Locally around points in $[0,\infty)\cap G$ the radial function increases on $G$ when following the orientation. An example is provided in Figure~\ref{fig:genericoutward}.
 \begin{figure}[h]
\centering
\def\svgscale{0.8}
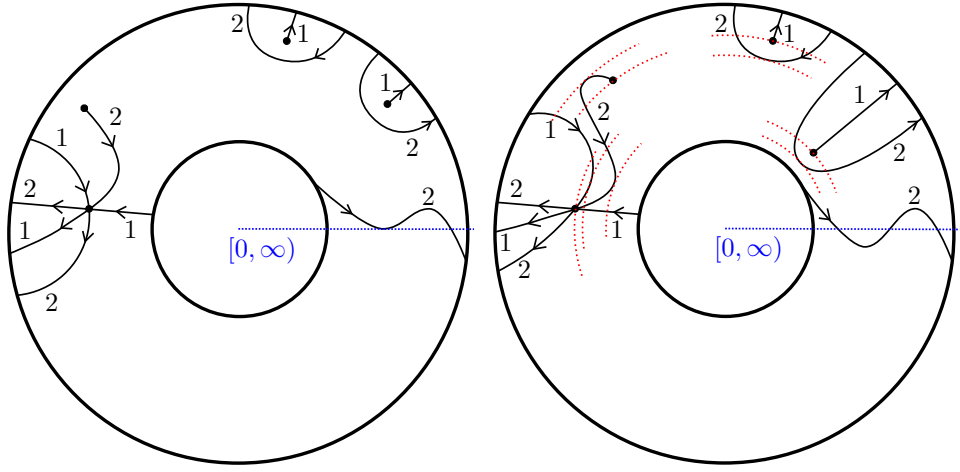
\caption{A Rudolph diagram in the annulus (left) is arranged to be outward-oriented (right). The exceptional circles are indicated by circle segments (red).}
\label{fig:genericoutward}
\end{figure}

\end{Remark}

Given an outward-oriented Rudolph diagram $G$ on $A$,
let $r_1<\cdots<r_k$ denote the radii corresponding to exceptional circles or points where $G$ meets $[0,\infty)$.
For $s$ in $[1,2]\setminus\{r_1,\cdots r_k\}$, we denote by $\beta_s$ the braid read off when following the counter-clockwise oriented circle of radius $s$ with $s$ as marked starting point.
We associate to $G$ the following finite sequence of braid words
\begin{equation}
\label{eq:assobraidseq}
(\beta_{s_0},\beta_{s_1},\cdots,\beta_{s_k}),\end{equation} where
\[1\leq s_0< r_1<s_1<r_2<s_2<\cdots<r_k<s_{k}\leq 2.\]
In particular, the sequence~\eqref{eq:assobraidseq} starts and ends with $\beta_1$ and $\beta_2$, respectively.% the braids obtained by following the boundary of $A$ counter-clockwise.
\begin{Obs}\label{Obs:consecbraids}
For all $0\leq l< k$, the braid word $\beta_{s_{l+1}}$ is obtained from $\beta_{s_{l}}$ by one of the following operations,
for some $1\leq i,j\leq n-1$ with $|i-j|\geq2$:
\begin{enumerate}[(i)]
\item\label{i} adding or removing subwords $a_ia_i^{-1}$ or $a_i^{-1}a_i$;
\item\label{ii} performing one braid relation; i.e.~replacing $a_ia_{i\pm1}a_i$ by $a_{i\pm1}a_{i}a_{i\pm1}$
or replacing $a_ia_j$ by $a_ja_i$;
\item\label{iii} %a cyclic permutation, i.e.~
changing a braid word of the form $a_i\alpha$ to $\alpha a_i$ or vice versa;
\item\label{iv} adding $a_i$ somewhere in the braid word.
\end{enumerate}

\end{Obs}
\begin{Remark}\label{Rem:consecbraids} Note that two braid words can be connected with a sequence using~\eqref{i} and~\eqref{ii} if and only if they define the same braid.
Two braid words can be connected with a sequence using~\eqref{i},~\eqref{ii}, and~\eqref{iii} if and only if they define the same closed braid.
And two braid words can be connected with a sequence using~\eqref{i} through~\eqref{iv} if and only if they are connected as described in Lemma~\ref{lemma:algrealization}.
\end{Remark}
Conversely, a sequence of braid words as described in Observation~\ref{Obs:consecbraids} yields a Rudolph diagram on $A$. This amounts to the following:
\begin{prop}\label{prop:assobraidseq}
The assignment given by~\eqref{eq:assobraidseq} yields an one-to-one correspondence
between outward-oriented Rudolph diagrams on $A$, up to isotopy through outward-oriented Rudolph diagrams,
and finite non-empty sequences of braid words $(\beta_0,\cdots,\beta_k)$ such that $\beta_{j+1}$ is obtained from $\beta_j$ by one of the operations~\eqref{i},~\eqref{ii},~\eqref{iii} and~\eqref{iv} described in Observation~\ref{Obs:consecbraids}. %by one of the following operations:
%\begin{itemize}
%\item adding or removing subwords $a_ia_i^{-1}$ or $a_i^{-1}a_i$,
%\item performing one braid relation,
%\item conjugate by a generator.
%\item adding a positive generator $a_i$ as a subword.
%\end{itemize}
\end{prop}
%Proposition~\ref{prop:assobraidseq} follows immediately from the way we arranged things in Remark~\ref{Rem:genericoutwarddiagram}.
\subsection{Smoothing of Rudolph diagrams and proof of Lemma~\ref{lemma:algrealization}}
After this translation of sequences of braid words to Rudolph diagrams we need a final ingredient to prove Lemma~\ref{lemma:algrealization}:
\begin{prop}\label{prop:smoothingofRD}
Let $G$ be a Rudolph diagram on $A$. There exists a smooth Rudolph diagram $\widetilde{G}$ on $A$ satisfying the following:

\noindent\textbullet\; $G$ and $\widetilde{G}$ are identical in a neighborhood of the inner boundary $S^1$. In particular, the braid words $\beta_1$ and $\widetilde{\beta}_1$ corresponding to $S^1$ %read off when following $S^1$ %the inner boundary
are the same.

\noindent\textbullet\; The braids $\beta_2$ and $\widetilde{\beta}_2$ corresponding to %read off when following
the outer boundary $S^1_2$
have the same closure.

\end{prop}
\begin{proof}
Let $B=\{v_1,\cdots,v_k\}$ be the set of $1$-valent vertices in $G$.
For every $v$ in $B$, we choose an embedded arc $\gamma_v$ in $A$ that connects $v$ to the inner boundary $S^1$ of $A$ such that $\gamma_v$ intersects $G$ in the interior of $A$ transversally and outside of vertices (except at $v$ of course). Furthermore, we arrange that all the arcs $\gamma_{v_1},\cdots,\gamma_{v_k}$ are pairwise disjoint. A neighborhood $N$ in $A$ of the union $S^1\cup\gamma_{v_1}\cup\cdots\cup\gamma_{v_k}$ defines an annulus on which $G$ is smooth. The boundary of $N$ has two components: $S^1$ and a curve that is isotopic to $S^1_2$ in $A\setminus B$. Therefore, we obtain a Rudolph diagram on $A$ as wanted by identifying $N$ with $A$ via a diffeomorphism that is the identity in a neighborhood of $S^1$. This is illustrated in Figure~\ref{fig:SmoothingRD}.
\begin{figure}[h]
\centering
\def\svgscale{1.4}
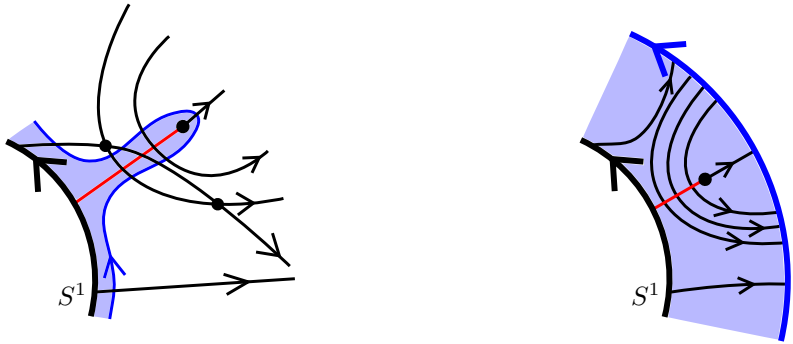
\caption{Left: A neighborhood $N$ (blue) of the inner boundary $S^1$ of $A$ and the embedded arc $\gamma_v$ (red). Right: Restriction of the Rudolph diagram to %the neighborhood
$N$, where $N$ is identified with~$A$.}
\label{fig:SmoothingRD}
\end{figure}
\end{proof}
\begin{proof}[Proof of Lemma~\ref{lemma:algrealization}]
Let $\beta_1$ and $\beta_2$ be quasi-positive braid words satisfying the assumptions in Lemma~\ref{lemma:algrealization}.
We find a finite sequence of braid words as described in Observation~\ref{Obs:consecbraids} that starts with $\beta_1$ and ends with $\beta_2$. %We may arrange for $\beta_1$ to be reduced.
Let $G$ be the corresponding Rudolph diagram on $A$ provided by Proposition~\ref{prop:assobraidseq}. By Proposition~\ref{prop:smoothingofRD} we may assume that $G$ is smooth %and has no arcs starting and ending on the inner boundary $S^1$
(this may change $\beta_2$ but the corresponding closed braid remains the same). %There is a smooth Rudolph diagram $G^s$ on $A$ with out closed circles which has the same closed braids corresponding two the boundary curves as $G$
%and in fact is equal to $G$ in a neighborhood of $S^1$, see ?.
Since $\beta_1$ is quasi-positive, there is a smooth Rudolph diagram $\widetilde{G}$ %with out closed cycles
on $D$ such that $\beta_1$ is the braid word read off when following the boundary of $D$ by Example~\ref{Ex:realizingQPasRD}.
%The Rudolph diagram $\widetilde{G}$ can be chosen to be smooth, compare or do a similar argument as in the proof of Proposition~\ref{prop:smoothingofRD}.
We glue $(\widetilde{G},D)$ and $(G,A)$ together along $S^1$ to get (by rescaling) a smooth Rudolph diagram $R$ on the disk $D$;
see left-hand side of Figure~\ref{fig:uniformization}.
\begin{figure}[h]
\centering
\def\svgscale{1.4}
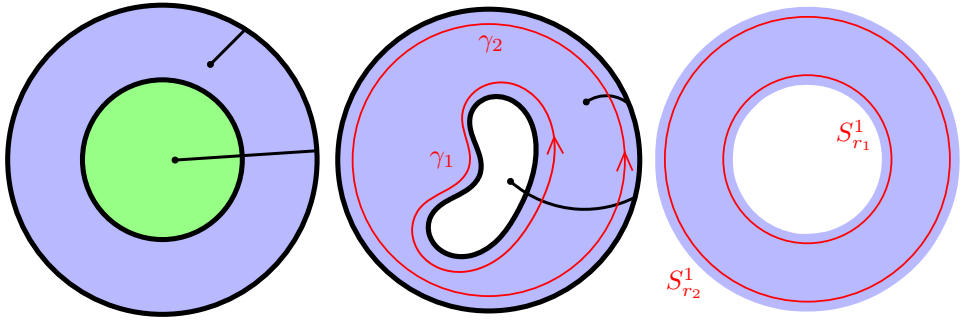
\caption{Left: The Rudolph diagram $R$ in $D$, which is obtained by gluing $(\widetilde{G},D)$ (green) and $(G,A)$ (blue) together.
Middle: Realization of $R$ in $D$ as $G(f)$ with the isotoped annulus $A$ (blue) and the curves $\gamma_1$ and $\gamma_2$ (red). Right: The annulus $\widetilde{A}$ with $S^1_{r_1}$ and $S^1_{r_1}$ (red).}
\label{fig:uniformization}
\end{figure}
Next, we remove all closed cycles in $R$. This might change the braid word $\beta_1$ but not the closed braid it represents by Remark~\ref{Rem:cycleremove}.
By Proposition~\ref{prop:RudolphRealization}, there exists an algebraic function $f$ such that $R=G(f)$ after an isotopy of $R$.
The latter isotopy yields an embedding of $A$ in $\C$ (we denote its image again by $A$) such that $G(f)\cap A=G$.
By the uniformization theorem for open annuli (see e.g.~\cite[6.4.~Theorem~10]{Ahlfors_IntrodToAnalyticFunctions}), there exists a biholomorphic map $\phi$ from the interior of $A\subset\C$ to an open annulus $\widetilde{A}\subset\C$ with concentric boundary circles centered at the origin.
Setting $\widetilde{f}(x,y)=f(\phi^{-1}(x),y)$ defines holomorphic map on $\widetilde{A}\times\C$ of the form \[y^n+\widetilde{{c}_{n-1}}(x)y^{n-1}+\cdots+\widetilde{c_0}(x) \text{ with } \widetilde{c_i}\text{ holomorphic on } \widetilde{A}.\]
We choose two concentric circles $S^1_{r_1}$ and $S^1_{r_2}$ in $\widetilde{A}$ such that their preimages under $\phi$ are curves $\gamma_1$ and $\gamma_2$ in $A$ for which the closed braids $V_f\pitchfork(\gamma_i\times\C)$ are $\overline{\beta_i}$. This is, for example, achieved by choosing $r_1$ and $r_2$ close to the radii of the inner and the outer boundary of $\widetilde{A}$, respectively; see Figure~\ref{fig:uniformization}. Therefore,
%such that $G(\widetilde{f})$ is isotopic to $G$.
%We remark that $G(\widetilde{f})$ can be defined on an open subset $U\subset\C$
%when the $c_i(x)$ are holomorphic functions on $U$, in the same way as in the algebraic setting in Section~\ref{subsec:RD}.
%In particular, we can choose two concentric circles $S^1_{r_1}$ and $S^1_{r_2}$ in $\widetilde{A}$ such that
$V_{\widetilde{f}}$ intersects the cylinders $Z_i=S^1_{r_i}\times\C$ transversely in closed braids and those closed braids are $\overline{\beta_i}$ since $V_{\widetilde{f}}\pitchfork(S^1_{r_i}\times\C)$ is the image of $V_f\pitchfork(\gamma_i\times\C)$ under the biholomorphic map
\[A\times\C\to \widetilde{A}\times\C,(x,y)\mapsto (\phi(x),y).\]
By polynomial approximation of the holomorphic coefficients $\widetilde{c_i}$ of $\widetilde{f}$, we find a polynomial
$g=y^n+{c}_{n-1}(x)y^{n-1}+\cdots+c_0(x)$ with $c_i\in\C[x]$ such that its zero-set $V_g$ intersects the above cylinders transversally in the same closed braids as $V_{\widetilde{f}}$.
We replace the cylinders $Z_i$ with cylinders ${Z_{i,R}}=\{x,y\in\C\;|\;|x|^2+\frac{|y|}{R}^2=r_i^2,x\neq0\}$, which for
large enough $R$ intersect $V_g$ in the same closed braids as the $Z_i$. Finally, we set $F(x,y)=\frac{1}{R^n}g(x,Ry)$ and conclude the proof by noticing that the $3$-spheres $S^3_{r_i}$ of radius $r_i$ intersect the zero-set $V_F$ in the links that are the closures of $\beta_i$ since rescaling the $y$-coordinate by the factor $\frac{1}{R}$ maps $Z_{i,R}$ onto $S^3_{r_i}\setminus(\{0\}\times{S^1_{r_i}})$.

\end{proof}

\section{Construction of optimal and algebraic cobordisms between torus knots %of small braid index
via positive braids}\label{sec:constructionofcobs}
In this section, we construct several families of optimal cobordisms between torus knots, which are also algebraic by Lemma~\ref{lemma:algrealization}.
It came as a surprise to the author that Ozsv\'ath, Stipsicz, and Szab\'o's $\Upsilon$-invariant shows that the constructions for torus knots of braid index 4 or less cannot be improved. %In a first approach we used obstructions coming from the classical signature as introduced by Trotter~\cite{Trotter_62_HomologywithApptoKnotTheory} and its generalization as introduced by Levine and Tristram~\cite{levine}\cite{tristram}. %However, the signatures are not strong enough to prove Corollary~\ref{cor:cobdist}.
\begin{Definition}
For links $K$ and $T$ that are closures of positive braids,
we say $K$ is \emph{subword-adjacent} to $T$, denoted by $K\leq_s T$,
if there are positive $n$-braids $\beta_1$ and $\beta_2$, %in $B_{n,+}\subset B_n$
for some integer $n$, such that $\beta_1$ can be obtained from $\beta_2$ by successively deleting generators.
\end{Definition}
Here, deleting a generator in a positive braid $\beta$ means removing an $a_i$ in a positive braid word that represents $\beta$.
We think of subword-adjacency as a combinatorial toy model for adjacency of singularities (as described in Section~\ref{sec:motivation}), hence the name.
\begin{Remark}\label{Rem:subwordadjacency}
If a positive $n$-braid $\beta_1$ is obtained from a positive $n$-braid $\beta_2$ by deleting positive generators, then $\beta_2$ can be obtained from $\beta_1$ as described in Lemma~\ref{lemma:algrealization}. %; however, staying in the positive braid semigroup $B_{n,+}$ at every step.
%In particular, $F_{\beta_2}$ can be obtained from $F_{\beta_1}$ by attaching $\chi(F_{\beta_1})-\chi(F_{\beta_2})=\l(\beta_2)-\l(\beta_1)$ $1$-handles in $S^3$, which yields a cobordism between $K$ and $T$ with Euler characteristic $\chi(F_{\beta_2})-\chi(F_{\beta_1})$.
Therefore, if $K$ is subword-adjacent to $T$, then there exists an algebraic cobordism between them, % which %is %even
%an algebraic cobordism,
by Lemma~\ref{lemma:algrealization}. %, there exists an algebraic cobordism between them.
\end{Remark}
%In terms of Rudolph diagrams: A subword-adjacent between two links $K$ and $T$ corresponds %positive braids $\beta_1$ and $\beta_2$ that arise as the boundary braids of
%a Rudolph diagram $G$ on the annulus $A$ without radial maxima or minima such that the two braids $\beta_1$ and $\beta_2$ given by $G$ have closure $K$ and $T$, respectively.
%We note that $\beta_1$ and $\beta_2$ are precisely as in Lemma~\ref{lemma:algrealization}, except that additionally at every step we stay in $B_{n,+}$.

\begin{Remark}\label{Rem:PosDestab}
In what follows %we fix a positive integer $n$\textemdash the number of strands\textemdash and
we consider positive braids $\beta$ %in $B_n$
with non-split closure; in particular, their closures $\overline{\beta}$ are fibered; see Stallings~\cite{Stallings_78_ConstructionsOfFibredKnotsandLinks}. %The fiber surface is $F_{\beta}$, which is honestly embedded in $S^3$ (instead of just a ribbon embedding).
%, we denote with fiber surfaces $F$.
In this case, the optimal cobordism provided by a subword-adjacency can be understood on the fiber surfaces:
Removing a generator in a positive braid $\beta$ corresponds to deplumbing a positive Hopf band on its fiber surface $F_\beta$.
In other words, if $K=\overline{\beta_1}$ is subword-adjacent to $T=\overline{\beta_2}$,
then the open book of $S^3$ with binding $K$ is obtained from the open book of $S^3$ with binding $T$ by $\chi(F_{\beta_1})-\chi(F_{\beta_2})=\l(\beta_2)-\l(\beta_1)$ positive destabilizations.
\end{Remark}
In this section, we use fence diagrams to represent positive braids.
%Positive braids can be graphically represented as fence diagrams.
I.e.~in a braid diagram, positive crossings $~\xygraph{
!{0;/r0.7pc/:}
[u(0.5)]%{\vdots}
!{\vcrossneg}
}\;$ are replaced with horizontal intervals
$~\xygraph{
!{0;/r0.7pc/:}
[u(0.5)]!{\xcapv[1]@(0)}
[lu]!{\xcapv[1]@(0)}
[u(0.5)] !{\xcaph[1]@(0)}}\hspace{-0.5pc}$;
see Rudolph~\cite{Rudolph_98_QuasipositivePlumbing}.
For example, the positive 3-braid $ a_1 a_2$ is represented by
$\xygraph{
!{0;/r0.7pc/:}
[u(1)]!{\xcapv[1.5]@(0)}
[lu]!{\xcapv[1.5]@(0)}
[lu]!{\xcapv[1.5]@(0)}
[u(0.5)r] !{\xcaph[1]@(0)}
[d(0.5)ll] !{\xcaph[1]@(0)}
%[lll] !{\xcaph[1]@(0)}
%[lll] !{\xcaph[1]@(0)}
}\hspace{-1pc}$
instead of the braid diagram $~\xygraph{
!{0;/r0.7pc/:}
[u(2.5)]%{\vdots}
[d(1.4)]!{\vcrossneg}
[ul]!{\xcapv[1]@(0)} [ld]
[ur]!{\vcrossneg}
[urr]!{\xcapv[1]@(0)} [ld]
%[u]!{\vcrossneg}
%[ul]!{\xcapv[1]@(0)}[dr]
%[urr]!{\xcapv[1]@(0)} [ld]
%{\vdots}
}\hspace{-0.5pc}$.

Simple examples of a subword-adjacency, which yield well-known optimal cobordisms, are the following.
\begin{Example}\label{Ex:incompTnm<Tabifn<am<b}
 Let $n,m,a,b$ be positive integers. If $n\leq a$ and $m\leq b$, then $T_{n,m}$ is subword-adjacent to $T_{a,b}$.
%, compare e.g.~Baader~\cite{Baader_ScissorEq}
%In particular, %if a positive torus knot $T_{n,m}$ is `obviously' simpler than a torus knot $T_{a,b}$, then there
%exists an optimal cobordism between them.
This subword-adjacency is obtained by deleting generators in the positive $a$-braid word $( a_1\cdots a_{a-1})^b$,
which has closure $T_{a,b}$, until one reaches a positive braid word with closure $T_{n,m}$.
Figure~\ref{fig:subwordadjacency} illustrates this for the adjacency $T(4,5)\leq_s T(7,7)$.
\begin{figure}[h]
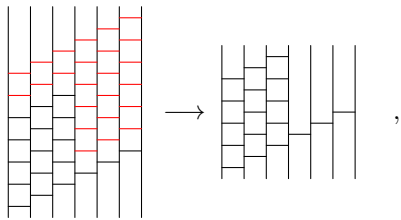

\centering
\[~\xygraph{
!{0;/r0.7pc/:}
[u(4.75)]!{\xcapv[9.5]@(0)}
[lu]!{\xcapv[9.5]@(0)}
[lu]!{\xcapv[9.5]@(0)}
[lu]!{\xcapv[9.5]@(0)}
[lu]!{\xcapv[9.5]@(0)}
[lu]!{\xcapv[9.5]@(0)}
[lu]!{\xcapv[9.5]@(0)}
[d(2)] !{\color{red}\xcaph[1]@(0)}
[u(0.5)] !{\xcaph[1]@(0)}
[u(0.5)] !{\xcaph[1]@(0)}
[u(0.5)] !{\xcaph[1]@(0)}
[u(0.5)] !{\xcaph[1]@(0)}
[u(0.5)] !{\xcaph[1]@(0)}
[d(3.5)llllll] !{\xcaph[1]@(0)}
[u(0.5)] !{\xcaph[1]@(0)}
[u(0.5)] !{\xcaph[1]@(0)}
[u(0.5)] !{\xcaph[1]@(0)}
[u(0.5)] !{\xcaph[1]@(0)}
[u(0.5)] !{\xcaph[1]@(0)}
[d(2)lll] !{\xcaph[1]@(0)}
[u(0.5)] !{\xcaph[1]@(0)}
[u(0.5)] !{\xcaph[1]@(0)}
[d(2)lll] !{\xcaph[1]@(0)}
[u(0.5)] !{\xcaph[1]@(0)}
[u(0.5)] !{\xcaph[1]@(0)}
[d(2)lll] !{\xcaph[1]@(0)}
[u(0.5)] !{\xcaph[1]@(0)}
[u(0.5)] !{\xcaph[1]@(0)}
[d(2)lll] !{\xcaph[1]@(0)}
[u(0.5)] !{\xcaph[1]@(0)}
[u(0.5)] !{\xcaph[1]@(0)}
[d(-0.5)llllll] !{\color{black}\xcaph[1]@(0)}
[u(0.5)] !{\xcaph[1]@(0)}
[u(0.5)] !{\xcaph[1]@(0)}
%[u(0.5)] !{\color{red}\xcaph[1]@(0)}
%[u(0.5)] !{\xcaph[1]@(0)}
%[u(0.5)] !{\xcaph[1]@(0)}
[d(2)lll] !{\color{black}\xcaph[1]@(0)}
[u(0.5)] !{\xcaph[1]@(0)}
[u(0.5)] !{\xcaph[1]@(0)}
%[u(0.5)] !{\color{red}\xcaph[1]@(0)}
%[u(0.5)] !{\xcaph[1]@(0)}
%[u(0.5)] !{\xcaph[1]@(0)}
[d(2)lll] !{\color{black}\xcaph[1]@(0)}
[u(0.5)] !{\xcaph[1]@(0)}
[u(0.5)] !{\xcaph[1]@(0)}
%[u(0.5)] !{\color{red}\xcaph[1]@(0)}
%[u(0.5)] !{\xcaph[1]@(0)}
%[u(0.5)] !{\xcaph[1]@(0)}
[d(2)lll] !{\color{black}\xcaph[1]@(0)}
[u(0.5)] !{\xcaph[1]@(0)}
[u(0.5)] !{\xcaph[1]@(0)}
%[u(0.5)] !{\xcaph[1]@(0)}
%[u(0.5)] !{\xcaph[1]@(0)}
%[u(0.5)] !{\xcaph[1]@(0)}
[d(2)lll] !{\color{black}\xcaph[1]@(0)}
[u(0.5)] !{\xcaph[1]@(0)}
[u(0.5)] !{\xcaph[1]@(0)}
[u(0.5)] !{\xcaph[1]@(0)}
[u(0.5)] !{\xcaph[1]@(0)}
[u(0.5)] !{\xcaph[1]@(0)}
}\hspace{-6.2pc}
\longrightarrow
~\xygraph{
!{0;/r0.7pc/:}
[u(3)]!{\xcapv[6]@(0)}
[lu]!{\xcapv[6]@(0)}
[lu]!{\xcapv[6]@(0)}
[lu]!{\xcapv[6]@(0)}
[lu]!{\xcapv[6]@(0)}
[lu]!{\xcapv[6]@(0)}
[lu]!{\xcapv[6]@(0)}
[d(0.5)] !{\xcaph[1]@(0)}
[u(0.5)] !{\xcaph[1]@(0)}
[u(0.5)] !{\xcaph[1]@(0)}
[d(2)lll] !{\xcaph[1]@(0)}
[u(0.5)] !{\xcaph[1]@(0)}
[u(0.5)] !{\xcaph[1]@(0)}
[d(2)lll] !{\xcaph[1]@(0)}
[u(0.5)] !{\xcaph[1]@(0)}
[u(0.5)] !{\xcaph[1]@(0)}
[d(2)lll] !{\xcaph[1]@(0)}
[u(0.5)] !{\xcaph[1]@(0)}
[u(0.5)] !{\xcaph[1]@(0)}
[d(2)lll] !{\xcaph[1]@(0)}
[u(0.5)] !{\xcaph[1]@(0)}
[u(0.5)] !{\xcaph[1]@(0)}
[u(0.5)] !{\xcaph[1]@(0)}
[u(0.5)] !{\xcaph[1]@(0)}
[u(0.5)] !{\xcaph[1]@(0)}
}\hspace{-3pc},
\]

\caption{Subword adjacency between $T_{4,5}$ and $T_{7,7}$. The arrow indicates the removal of the generators marked in red.}\label{fig:subwordadjacency} %The generalization to the general case is immediate.
\end{figure}\end{Example}
The subword-adjacencies given in Example~\ref{Ex:incompTnm<Tabifn<am<b}, have analogs in the algebraic adjacency setting since it is easy to write down an adjacency from $y^a-x^b$ to $y^n-x^m$.

A proposition due to Baader provides examples of subword-adjacencies that are more interesting.
\begin{prop}\label{prop:fibersurface:a,bc->b,ac}
\cite[Proposition~1]{Baader_ScissorEq} Let $a,b,c$ be positive integers with $a\leq b$.
Then $T_{a,bc}$ is subword-adjacent to $T_{b,ac}$.
\end{prop}
\color{black}However, again there exists an algebraic adjacency from $y^a-x^{bc}$ to $y^{b}-x^{ac}$; see~\cite[Proposition~23]{Feller_14_GordianAdjacency},
which yields an algebraic cobordisms from $T_{a,bc}$ to $T_{b,ac}$ without appealing to Lemma~\ref{lemma:algrealization}.

\subsection{Optimal examples for torus knots of small braid index and proofs of the main results}
After these first examples, we proceed with subword adjacencies between torus knots that turn out to be optimal and that, to the author's knowledge, are not known to have
algebraic adjacency analogs.

The following propositions provide all optimal cobordisms that are needed to establish Theorem~\ref{thm:32}, Theorem~\ref{thm:42}, and Corollary~\ref{cor:cobdist}.

\begin{prop}\label{prop:subsurfacesofIndex3toruslinks}
Let $n$ and $m$ be positive integers. If $n\leq \frac{5m-1}{3}$, then the torus link $T_{2,n}$ is subword-adjacent to the torus link $T_{3,m}$.
\end{prop}
\begin{prop}\label{prop:T42m>T25m}
 Let $n$ and $m$ be positive integers. If $n\leq \frac{5m-3}{2}$, then the torus link $T_{2,n}$ is subword-adjacent to the torus link $T_{4,m}$.
\end{prop}

It is part of the statement of Theorems~\ref{thm:32} and~\ref{thm:42}
that Propositions~\ref{prop:subsurfacesofIndex3toruslinks} and~\ref{prop:T42m>T25m} cannot be improved, at least when the involved links are knots. This is a consequence of the cobordism distance bound
\begin{equation}\label{eq:HFobstruction} d_c(K,T)= g_s(K\sharp-{T})\geq \max\{|\tau(K)-\tau(T)|,|\upsilon(K)-\upsilon(T)|\},\end{equation}
provided in~\cite[Theorem~1.11]{OSS_2014}, generalizing the $\tau$-bound in~\cite{OzsvathSzabo_03_KFHandthefourballgenus}.
Before proving Propositions~\ref{prop:subsurfacesofIndex3toruslinks} and~\ref{prop:T42m>T25m},
we use them and~\eqref{eq:HFobstruction} to prove Theorem~\ref{thm:32}, Theorem~\ref{thm:42}, and Corollary~\ref{cor:cobdist}.

\begin{proof}[Proof of Theorems~\ref{thm:32} and~\ref{thm:42}]
By Remark~\ref{Rem:subwordadjacency}, the fact that \eqref{2} implies %\eqref{1} and
\eqref{3} is an immediate consequence of Proposition~\ref{prop:subsurfacesofIndex3toruslinks} and Proposition~\ref{prop:T42m>T25m}, respectively. By the Thom conjecture, \eqref{3} implies \eqref{1}. Therefore, it remains to show that \eqref{1} implies~\eqref{2}.

Throughout the proof we have $K=T_{2,n}$, for some odd integer $n\geq3$, and $T$ is a torus knot of braid index 3 or 4.
We assume towards a contradiction that~\eqref{2} does not hold and calculate that \begin{equation}\label{eq:upsilonobstr} |g(T)-g(K)|<|\upsilon(T)-\upsilon(K)|,\end{equation} which contradicts~\eqref{1} by~\eqref{eq:HFobstruction}.
%Note that $\upsilon(K)=-g(K)$ and $\upsilon(T)>-g(T)$, see~\eqref{upsilonsfortk}.
%Therefore, if $g(K)>g(T)$, then~\eqref{eq:upsilonobstr} holds.
%Thus, we only have to consider cases where $g(K)\leq g(T)$.
We do this according to the 3 cases $T=T_{3,3k+1}$, $T=T_{3,3k+2}$ (Theorem~\ref{thm:32}); and $T=T_{4,2k+1}$ (Theorem~\ref{thm:42}), where $k$ is a positive integer:

For $T=T_{3,3k+2}$, we have that~\eqref{2} fails precisely when
\begin{align*}
n\geq 5k+4&\Longleftrightarrow 5k+2<n-1\\
&\Longleftrightarrow 3k+1-\frac{n-1}{2}<\frac{n-1}{2}-(2k+1)\\
&\Longleftrightarrow \left|3k+1-\frac{n-1}{2}\right|<\left|\frac{n-1}{2}-(2k+1)\right|\\
&\overset{\eqref{eq:TCforTK}\eqref{upsilonsfortk}}{\Longleftrightarrow} \left|g(T)-g(K)\right|<\left|-\upsilon(K)+\upsilon(T)\right|.\\
%&\Longleftrightarrow -g(T)+g(K)>+\upsilon(K)-\upsilon(T).
\end{align*}
This shows that, if \eqref{2} fails, then \eqref{eq:upsilonobstr} holds.

Similarly, for $T=T_{3,3k+1}$ and $T=T_{4,2k+1}$, we have that~\eqref{2} fails precisely
when
\begin{align*}
n\geq 5k+2&\Longleftrightarrow 5k< n-1\\
&\Longleftrightarrow 3k-\frac{n-1}{2}<\frac{n-1}{2}-2k\\
&\Longleftrightarrow \left|3k-\frac{n-1}{2}\right|<\left|\frac{n-1}{2}-2k\right|\\
&\overset{\eqref{eq:TCforTK}\eqref{upsilonsfortk}}{\Longleftrightarrow} \left|g(T)-g(K)\right|<\left|-\upsilon(K)+\upsilon(T)\right|.
\end{align*}
As before this shows that, if~\eqref{2} fails, then~\eqref{eq:upsilonobstr} holds.
\end{proof}
\begin{proof}[Proof of Corollary~\ref{cor:cobdist}]
For torus knots $K$ and $T$ such that the sum of their braid indices is $6$ or less, we want to establish
\[d_c(K,T)=g_s(K\sharp\overline{T})=\max\{|\tau(K)-\tau(T)|,|\upsilon(K)-\upsilon(T)|\}.\]
By~\eqref{eq:HFobstruction}, it suffices to find a cobordism $C$ between $K$ and $T$ with genus
\[g(C)\leq
\max\{|\tau(K)-\tau(T)|,|\upsilon(K)-\upsilon(T)|\}.\] %i.e.\ find a cobordism with Of course

If $K$ and $T$ are torus knots of opposite sign, then
we have that
\[d_c(K,T)=g_s(K\sharp\overline{T})=g(K)+g(T),\] where the second equality invokes the Thom Conjecture; compare Remark~\ref{Rem:cor:cobdist}.
%and the last equality uses the values of $\tau$ as calculated in~\cite[Corollary~1.7]{OzsvathSzabo_03_KFHandthefourballgenus}.
If $K$ and $T$ are positive torus knots that have the same braid index, then
\[d_c(K,T)=g_s(K\sharp\overline{T})=|g(K)-g(T)|,\] where the second equality follows from Example~\ref{Ex:incompTnm<Tabifn<am<b}; compare Remark~\ref{Rem:cor:cobdist}. % and the last equality uses the values of $\tau$ as calculated in~\cite[Corollary~1.7]{OzsvathSzabo_03_KFHandthefourballgenus}.
Therefore, if $K$ and $T$ are torus knots of opposite sign or torus knots with the same braid index, then
\[d_c(K,T)=g_s(K\sharp\overline{T})=|\tau(K)-\tau(T)|\]
since $g(K)+g(T)=|\tau(K)-\tau(T)|$ or $|g(K)-g(T)|=|\tau(K)-\tau(T)|$, respectively; compare ~\cite[Corollary~1.7]{OzsvathSzabo_03_KFHandthefourballgenus}. % by the values of $\tau$ for torus knots
Thus, after taking mirror images of $K$ and $T$, if necessary,
we may assume that $K$ and $T$ are both positive torus knots such that $K$ has braid index $2$ and $T$ has braid index 3 or 4.

Let $n$ and $k$ be the positive integers such that $K$ is $T_{2,n}$ and $T$ is $T_{3,3k+1}$, $T_{3,3k+2}$, or $T_{4,2k+1}$.
Furthermore, we do not need to consider the cases covered by Theorems~\ref{thm:32} and~\ref{thm:42};
i.e.~when~\eqref{1},~\eqref{2}, and~\eqref{3} of Theorems~\ref{thm:32} and~\ref{thm:42}, respectively, are satisfied.
%\[d_c(K,T)=g(T)-g(K)=-\tau(T)+\tau(K)\overset{\eqref{upsilonsfortk}}{=}\max\{|\tau(K)-\tau(T)|,|\upsilon(K)-\upsilon(T)|\}.\]

We first consider $T=T_{3,3k+2}$. By Proposition~\ref{prop:subsurfacesofIndex3toruslinks} there exist a cobordism $C_1$
from $T$ to $T_{2,5k+3}$ which is optimal; i.e.~$C_1$ has Euler characteristic \[\chi_s(T)-\chi_s(T_{2,5k+3})=(-6k-1)-(-5k-1)=-k.\]
Note that $n>5k+3$ since we are assuming that there does not exist an optimal cobordism from $K$ to $T$; compare~\eqref{2} in Theorem~\ref{thm:32}.
By Example~\ref{Ex:incompTnm<Tabifn<am<b}, we have a cobordism $C_2$ from $T_{2,5k+3}$ to $K$ of Euler characteristic
\[\chi_s(K)-\chi_s(T_{2,5k+3})=-n+2-(-5k-1)=-n+5k+3.\]
Gluing $C_1$ and $C_2$ together yields a cobordism $C$ of
Euler characteristic $-n+4k+3$ between $T$ and $K$. Thus, \[g(C)=\frac{n-4k-3}{2}=\frac{n-1}{2}-2k-1\overset{\eqref{upsilonsfortk}}{=}-\upsilon(K)+\upsilon(T)=|\upsilon(K)-\upsilon(T)|.\]

Similarly, if $T=T_{3,3k+1}$ or $T=T_{4,2k+1}$,
we get a cobordism %$C_1$
from $T$ to $T_{2,5k+1}$ with Euler characteristic \[-6k+1-(-5k+1)=-k\] by Proposition~\ref{prop:subsurfacesofIndex3toruslinks} or Proposition~\ref{prop:T42m>T25m}, respectively,
and a cobordism %$C_2$
from $T_{2,5k+1}$ to $K$ with Euler characteristic \[-n+2-(-5k+1)=-n+5k+1.\] As before, we combine these two cobordisms to a cobordism $C$ from $T$ to $K$ with
\[g(C)=\frac{ (n-1)-4k}{2}=\frac{n-1}{2}-2k\overset{\eqref{upsilonsfortk}}{=}-\upsilon(K)+\upsilon(T)=|\upsilon(K)-\upsilon(T)|.\]

\end{proof}
We now provide proofs of Proposition~\ref{prop:subsurfacesofIndex3toruslinks} and Proposition~\ref{prop:T42m>T25m}. We thank Sebastian Baader for important inputs for these proofs.
\begin{proof}[Proof of Proposition~\ref{prop:subsurfacesofIndex3toruslinks}]
We denote the $3$-strand full twist $( a_1 a_2 a_1)^{2}$ by $\Delta^2$. The full twist commutes with every other $3$-braid.
% =~\xygraph{
% !{0;/r0.7pc/:}
% [u(1.75)]!{\xcapv[3.5]@(0)}
% [lu]!{\xcapv[3.5]@(0)}
% [lu]!{\xcapv[3.5]@(0)}
% %[lu]!{\xcapv[2.5]@(0)}
% [u(0.5)] !{\xcaph[1]@(0)}
% [d(0.5)] !{\xcaph[1]@(0)}
% [d(0.5)ll] !{\xcaph[1]@(0)}
% [d(0.5)l] !{\xcaph[1]@(0)}
% [d(0.5)] !{\xcaph[1]@(0)}
% [d(0.5)ll] !{\xcaph[1]@(0)}
% %[lll] !{\xcaph[1]@(0)}
% }\hspace{-1.5pc}.$

Let us first consider the case where $m=3l$ for some positive integer $l$. %One needs to show $T(2,5l-1)\leq T(3,3l)$.
The torus link $T_{3,3l}$ is the closure of $3$-braid $\Delta^{2l}$. %, which is obtained by inductively adding full twists to full twists.
Note that
\[\Delta^2\Delta^2 =  a_1 a_2 a_1 a_1 a_2 a_1\Delta^2 =  a_1 a_2 a_1 a_1 a_2\Delta^2 a_1.\]
% =~\xygraph{
% !{0;/r0.7pc/:}
% [u(3.25)]!{\xcapv[6.5]@(0)}
% [lu]!{\xcapv[6.5]@(0)}
% [lu]!{\xcapv[6.5]@(0)}
% %[lu]!{\xcapv[2.5]@(0)}
% [u(0.5)] !{\xcaph[1]@(0)}
% [d(0.5)l] !{\xcaph[1]@(0)}
% [d(0.5)] !{\xcaph[1]@(0)}
% [d(0.5)ll] !{\xcaph[1]@(0)}
% [d(0.5)l] !{\xcaph[1]@(0)}
% [d(0.5)] !{\xcaph[1]@(0)}
% [d(0.5)ll] !{\xcaph[1]@(0)}
% [d(0.5)] !{\xcaph[1]@(0)}
% [d(0.5)ll] !{\xcaph[1]@(0)}
% [d(0.5)l] !{\xcaph[1]@(0)}
% [d(0.5)] !{\xcaph[1]@(0)}
% [d(0.5)ll] !{\xcaph[1]@(0)}
% %[lll] !{\xcaph[1]@(0)}
% }\hspace{-4pc}$
Adding another full twist yields
\[\Delta^{2}\Delta^{2}\Delta^{2} =  a_1 a_2 a_1 a_1 a_2\Delta^{2}\Delta^{2} a_1
=  a_1 a_2 a_1 a_1 a_2( a_1 a_2 a_1 a_1 a_2\Delta^2 a_1) a_1\]
and inductively we get $\Delta^{2l}=( a_1 a_2 a_1 a_1 a_2)^l( a_1)^l$. %, which has $T(3,3l)$ as closure.
The subword $ a_2 a_1 a_2$ occurs $l-1$ times in $( a_1 a_2 a_1 a_1 a_2)^l( a_1)^l$.
Applying $l-1$ times the braid relation $ a_2 a_1 a_2= a_1 a_2 a_1$ gives the $3$-braid word
\[w= a_1 a_2 a_1 a_1( a_1 a_2 a_1 a_1 a_1)^{l-1} a_2( a_1)^l.\]
Deleting all but the first $ a_2$ in $w$ yields $ a_1 a_2 a_1^{5l-2}$, which has $T_{2,5l-1}$ as closure.

If $m$ is $3l+1$ for some positive integer $l$, %one needs to show $T(2,5l+1)\leq T(3,3l+1)$.
we write $T_{3,3l+1}$ as the closure of %Similarly to the first case, we write
\begin{align*}
%T(3,3l+1)=
 a_1 a_2\Delta^{2l} &=  a_1 a_2 w = a_1 a_2 \left( a_1 a_2 a_1 a_1( a_1 a_2 a_1 a_1 a_1)^{l-1} a_2( a_1)^l\right)\\ %= a_1 a_2( a_1 a_2 a_1 a_1 a_2)^l( a_1)^l
&= a_1 a_1 a_2 a_1 a_1 a_1( a_1 a_2 a_1 a_1 a_1)^{l-1} a_2( a_1)^l.
%=  a_1( a_1 a_2 a_1 a_1 a_1)^l a_2( a_1)^l.
\end{align*}
%Applying $l$ times the braid relation and then
Deleting all but the first $ a_2$ yields
$ a_1 a_1 a_2( a_1)^{5l-1}$, which has closure $T_{2,5l+1}$.

Finally, if $m$ is $3l+2$ for some positive integer $l$, %we show $T(2,5l+3)\leq T(3,3l+2)$. For this
view $T_{3,3l+2}$ as the closure of
\begin{align*}
%&T(3,3l+2)&
%&=&
 a_1 a_2 a_1 a_2\Delta^{2l}&=
 a_1 a_1 a_2 a_1\Delta^{2l}
 =  a_1 (a_1 a_2\Delta^{2l}) a_1
\\
& = a_1 \left(a_1 a_1 a_2 a_1 a_1 a_1( a_1 a_2 a_1 a_1 a_1)^{l-1} a_2( a_1)^l\right) a_1.
\end{align*}
Deleting all but the first $ a_2$ yields
 $ a_1 a_1 a_1 a_2( a_1)^{5l}$, which has $T_{2,5l+3}$ as closure.
\end{proof}
\begin{proof}[Proof of Proposition~\ref{prop:T42m>T25m}]
 We view $T_{4,2l+1}$ as the closure of the 4-braid $( a_1 a_3 a_2)^{2l+1}$.
Using the fact that the half twist on 4 strands \[\Delta= a_1 a_3 a_2 a_1 a_3 a_2%= a_1 a_3 a_2 a_3 a_1 a_2
%= a_1 a_2 a_3 a_2 a_1 a_2= a_1 a_2 a_3 a_1 a_2 a_1
= a_1 a_2 a_1 a_3 a_2 a_1\] anti-commutes with every other 4-braid,
i.e.\  $ a_1\Delta=\Delta a_3$, $ a_3\Delta=\Delta a_1$, and $ a_2\Delta=\Delta a_2$,
we have that \begin{align*}
\Delta^k &=  a_1 a_2 a_1 a_3 a_2 a_1\Delta^{k-1} =  a_1 a_2 a_1 a_3 a_2\Delta^{k-1} a_{2+{(-1)}^{k-1}}\\
&=  a_1 a_2 a_1 a_3 a_2\left( a_1 a_2 a_1 a_3 a_2\Delta^{k-2} a_{2+{(-1)}^{k-2}}\right) a_{2+{(-1)}^{k-1}}\\
&= ( a_1 a_2 a_1 a_3 a_2)^2\Delta^{k-2} a_{1} a_{3} = \cdots
= ( a_1 a_2 a_1 a_3 a_2)^{k} a_1^{\lfloor \frac{k+1}{2}\rfloor} a_3^{\lfloor \frac{k}{2}\rfloor}.
\end{align*}
With this we can write $( a_1 a_3 a_2)^{2l+1}$ as follows.
\begin{align*}
( a_1 a_3 a_2)^{2l+1}&=( a_1 a_3 a_2)^{2} a_1 a_3 a_2( a_1 a_3 a_2)^{2(l-1)}
\\
&=( a_1 a_3 a_2)^{2} a_1 a_3 a_2( a_1 a_2 a_1 a_3 a_2)^{l-1} a_1^{\lfloor \frac{l}{2}\rfloor} a_3^{\lfloor \frac{l-1}{2}\rfloor}
%Applying the braid relation to change $l-1$ occurrence of $ a_2 a_1 a_2$ to $ a_2 a_2 a_1$ yields
\\
&=( a_1 a_3 a_2)^{2} a_1 a_3( a_2 a_1 a_2 a_1 a_3)^{l-1} a_2 a_1^{\lfloor \frac{l}{2}\rfloor} a_3^{\lfloor \frac{l-1}{2}\rfloor}
\\
&=( a_1 a_3 a_2)^{2} a_1 a_3( a_1 a_2 a_1 a_1 a_3)^{l-1} a_2 a_1^{\lfloor \frac{l}{2}\rfloor} a_3^{\lfloor \frac{l-1}{2}\rfloor}.
\end{align*}
Deleting the last $l$ occurrences of $ a_2$ in this braid word gives
\[( a_1 a_3 a_2)^{2} a_1 a_3( a_1 a_1 a_1 a_3)^{l-1} a_1^{\lfloor \frac{l}{2}\rfloor} a_3^{\lfloor \frac{l-1}{2}\rfloor}
=( a_1 a_3 a_2)^{2} a_1^{3l-2+\lfloor \frac{l}{2}\rfloor} a_3^{l+\lfloor \frac{l-1}{2}\rfloor},\]
which has closure
\[T_{2,4+(3l-2+\lfloor \frac{l}{2}\rfloor)+(l+\lfloor \frac{l-1}{2}\rfloor)}=T_{2,5l+1}.\]

Similarly, the torus link
$T_{4,2l+2}$ is the closure of the $4$-braid
\[( a_1 a_3 a_2)^{2l+2}=
( a_1 a_3 a_2)^{2} a_1 a_3 a_2 a_1 a_3( a_1 a_2 a_1 a_1 a_3)^{l-1} a_2 a_1^{\lfloor \frac{l}{2}\rfloor} a_3^{\lfloor \frac{l-1}{2}\rfloor}.\]
Deleting the last $l+1$ occurrences of $ a_2$ yields a $4$-braid that has closure $T_{2,5l+3}$.
\end{proof}

\subsection{Subword-adjacency for the torus knot $T_{m,m+1}$}\label{subsec:ddto2n}
We now study, which $T_{2,n}$ is subword-adjacent to $T_{m,m}$ and $T_{m,m+1}$.
Our result is roughly that, whenever $n\leq \frac{2m^2}{3}+O(m)$, then $T_{2,n}$ is subword-adjacent to $T_{m,m}$ and $T_{m,m+1}$.
Our interest in this stems from the fact that this is an improvement over what is known in the algebraic setting;
compare with% the bound in
~\eqref{eq:orevkov}.
In other words, the algebraic cobordism obtained by applying Lemma~\ref{lemma:algrealization} to the subword-adjacencies provided in the following Proposition is not known to come from an algebraic adjacency between $y^2-x^n$ and $y^m-x^m$ or a singular algebraic curve of degree $m$ with an $A_{n-1}$-singularity.
\begin{Remark}\label{rem:TmmtoT2n} After a first preprint of this article appeared, the author was pointed to work of Orevkov, where the same bound is attained in a very similar setting, and his result was also motivated by questions discussed in Section~\ref{sec:motivation}. Indeed, Orevkov's result~\cite[Theorem~3.13]{Orevkov_12_SomeExamples} allows to conclude that there exists an optimal cobordism between
$T_{2,n}$ and $T_{m,m}$, whenever $n\leq \frac{2m^2}{3}+O(m)$.
\end{Remark}

\begin{prop}\label{prop:subsurfacesofTmm}
 Let $m$ and $n$ be positive integers. If $n\leq \frac{2m^2+4}{3}-m$, then $T_{2,n}$ is subword-adjacent to $T_{m,m}$.
\end{prop}
\begin{Remark}\label{Rem:subsurfacesofTm(m+1)}
A similar statement holds for the knots $T_{m,m+1}$:
 let $m$ and $n$ be positive integers. If $n\leq \frac{2m^2-m+5}{3}$, then $T_{2,n}$ is subword-adjacent to $T_{m,m+1}$.
\end{Remark}
\begin{Remark}\label{Rem:subwordsofTmm}
We do not know whether the factor $\frac{2}{3}$ is optimal. If it is, the straight-forward application of $\Upsilon$ does not suffice to show this. In fact, it only gives us that, whenever there is an optimal cobordism between $T_{m,m+1}$ and $T_{2,n}$, then
\begin{equation}\label{eq:3/4m^2bound}n\leq \frac{3m^2}{4}+O(m),\end{equation}
which is the same upper bound that is known for the algebraic setting; see~\eqref{eq:asymbounds}.
Indeed, let us fix a positive integer $m$ and assume that there exists an optimal cobordism between the %knots
$T_{2,n}$ and $T_{m,m+1}$ for some odd $n>0$; i.e.~\[d_c(T_{2,n},T_{m,m+1})=|g(T_{2,n})-g(T_{m,m+1})|.\] Using
\begin{equation}\label{eq:upsilonformm+1}
\upsilon(T_{2,n})\overset{\text{\eqref{upsilonsfortk}}}{=}-\frac{n-1}{2}\et
\upsilon(T_{m,m+1})=-\left\lfloor \frac{m^2}{4}\right\rfloor\;\text{\cite[Proposition~6.3]{OSS_2014}},
\end{equation}
the obstruction given in \eqref{eq:HFobstruction} yields
\begin{align*}
-\upsilon(T_{2,n})+\upsilon(T_{m,m+1})&\leq g_s(T_{m,m+1})-g_s(T_{2,n})&\overset{\eqref{eq:TCforTK}\eqref{eq:upsilonformm+1}}{\Longleftrightarrow}\\
\frac{n-1}{2}-\left\lfloor\frac{ m^2}{4}\right\rfloor&\leq \frac{m(m-1)}{2}-\frac{n-1}{2}&\Longleftrightarrow \\n&\leq\frac{3m^2-2m+4}{4}\quad.&
\end{align*}
A similar calculation using the signature instead of $\upsilon$ also yields~\eqref{eq:3/4m^2bound}.
\end{Remark}

We proceed with the proof of Proposition~\ref{prop:subsurfacesofTmm}, which can be adapted to yield %a proof of
Remark~\ref{Rem:subsurfacesofTm(m+1)}.
\begin{proof}[Proof of Proposition~\ref{prop:subsurfacesofTmm}]

We denote by $\Delta_m$ the half twist on $m$ strands;
i.e.~the $m$-braid
\[( a_1 a_2\cdots a_{m-1})( a_1 a_2\cdots a_{m-2})\cdots ( a_1 a_2) a_1.\]
The torus link $T_{m,m}$ is the closure of the full twist on $m$ strands $\Delta_m^2$.

The main step in the proof consists of deleting generators in $\Delta_m$ yielding a braid that is a split union of positive $2$-braids and which has roughly length $\frac{2}{3}\l(\Delta_m)$.
%We first describe how to delete generators in $\Delta_m$ to obtain a braid $\beta_m$ that closes to a split union of $T(2,k)$ torus links.
More precisely, we delete the generator $ a_{m-1}$ in $\Delta_m$ and then apply braid relations to get the positive braid word
\[( a_1^2 a_2\cdots a_{m-2})\cdots ( a_1^2 a_2) a_1^2\text{ in }B_m.\]
Then, we delete all $ a_2$ yielding a split union of $ a_1^{2(m-2)}$ on strands 1 and 2, a half twist on the strands 3 to $m-1$, and strand $m$.
We illustrate this for $m=7$.%, generators that are removed are marked in red.
\begin{equation}\label{eq:Delta71}
\Delta_7=~\xygraph{
!{0;/r0.7pc/:}
[u(3)]!{\xcapv[6]@(0)}
[lu]!{\xcapv[6]@(0)}
[lu]!{\xcapv[6]@(0)}
[lu]!{\xcapv[6]@(0)}
[lu]!{\xcapv[6]@(0)}
[lu]!{\xcapv[6]@(0)}
[lu]!{\xcapv[6]@(0)}
[u(0.5)] !{\xcaph[1]@(0)}
[d(0.5)] !{\xcaph[1]@(0)}
[d(0.5)] !{\xcaph[1]@(0)}
[d(0.5)] !{\xcaph[1]@(0)}
[d(0.5)] !{\xcaph[1]@(0)}
[d(0.5)] !{\color{red}\xcaph[1]@(0)}
[u(1.5)llllll] !{\color{black}\xcaph[1]@(0)}
[d(0.5)] !{\xcaph[1]@(0)}
[d(0.5)] !{\xcaph[1]@(0)}
[d(0.5)] !{\xcaph[1]@(0)}
[d(0.5)] !{\xcaph[1]@(0)}
[u(1)lllll] !{\xcaph[1]@(0)}
[d(0.5)] !{\xcaph[1]@(0)}
[d(0.5)] !{\xcaph[1]@(0)}
[d(0.5)] !{\xcaph[1]@(0)}
[u(0.5)llll] !{\xcaph[1]@(0)}
[d(0.5)] !{\xcaph[1]@(0)}
[d(0.5)] !{\xcaph[1]@(0)}
[lll] !{\xcaph[1]@(0)}
[d(0.5)] !{\xcaph[1]@(0)}
[d(0.5)ll] !{\xcaph[1]@(0)}
}\hspace{-3.8pc}\longrightarrow
~\xygraph{
!{0;/r0.7pc/:}
[u(3)]!{\xcapv[6]@(0)}
[lu]!{\xcapv[6]@(0)}
[lu]!{\xcapv[6]@(0)}
[lu]!{\xcapv[6]@(0)}
[lu]!{\xcapv[6]@(0)}
[lu]!{\xcapv[6]@(0)}
[lu]!{\xcapv[6]@(0)}
[u(0.5)] !{\xcaph[1]@(0)}
[d(0.5)] !{\xcaph[1]@(0)}
[d(0.5)] !{\xcaph[1]@(0)}
[d(0.5)] !{\xcaph[1]@(0)}
[d(0.5)] !{\xcaph[1]@(0)}
[u(1)lllll] !{\xcaph[1]@(0)}
[d(0.5)] !{\xcaph[1]@(0)}
[d(0.5)] !{\xcaph[1]@(0)}
[d(0.5)] !{\xcaph[1]@(0)}
[d(0.5)] !{\xcaph[1]@(0)}
[u(1)lllll] !{\xcaph[1]@(0)}
[d(0.5)] !{\xcaph[1]@(0)}
[d(0.5)] !{\xcaph[1]@(0)}
[d(0.5)] !{\xcaph[1]@(0)}
[u(0.5)llll] !{\xcaph[1]@(0)}
[d(0.5)] !{\xcaph[1]@(0)}
[d(0.5)] !{\xcaph[1]@(0)}
[lll] !{\xcaph[1]@(0)}
[d(0.5)] !{\xcaph[1]@(0)}
[d(0.5)ll] !{\xcaph[1]@(0)}
}\hspace{-3.8pc}=
~\xygraph{
!{0;/r0.7pc/:}
[u(3.75)]!{\xcapv[7.5]@(0)}
[lu]!{\xcapv[7.5]@(0)}
[lu]!{\xcapv[7.5]@(0)}
[lu]!{\xcapv[7.5]@(0)}
[lu]!{\xcapv[7.5]@(0)}
[lu]!{\xcapv[7.5]@(0)}
[lu]!{\xcapv[7.5]@(0)}
[u(0.5)] !{\xcaph[1]@(0)}
[d(0.5)l] !{\xcaph[1]@(0)}
[d(0.5)] !{\color{red}\xcaph[1]@(0)}
[d(0.5)] !{\color{black}\xcaph[1]@(0)}
[d(0.5)] !{\xcaph[1]@(0)}
[d(0.5)] !{\xcaph[1]@(0)}
[u(1)lllll] !{\xcaph[1]@(0)}
[d(0.5)l] !{\xcaph[1]@(0)}
[d(0.5)] !{\color{red}\xcaph[1]@(0)}
[d(0.5)] !{\color{black}\xcaph[1]@(0)}
[d(0.5)] !{\xcaph[1]@(0)}
[u(0.5)llll] !{\xcaph[1]@(0)}
[d(0.5)l] !{\xcaph[1]@(0)}
[d(0.5)] !{\color{red}\xcaph[1]@(0)}
[d(0.5)] !{\color{black}\xcaph[1]@(0)}
[lll] !{\xcaph[1]@(0)}
[d(0.5)l] !{\xcaph[1]@(0)}
[d(0.5)] !{\color{red}\xcaph[1]@(0)}
[d(0.5)ll] !{\color{black}\xcaph[1]@(0)}
[d(0.5)l] !{\xcaph[1]@(0)}
}\hspace{-4.8pc}\longrightarrow
~\xygraph{
!{0;/r0.7pc/:}
[u(2.75)]!{\xcapv[5.5]@(0)}
[lu]!{\xcapv[5.5]@(0)}
[lu]!{\xcapv[5.5]@(0)}
[lu]!{\xcapv[5.5]@(0)}
[lu]!{\xcapv[5.5]@(0)}
[lu]!{\xcapv[5.5]@(0)}
[lu]!{\xcapv[5.5]@(0)}
[u(0.5)] !{\xcaph[1]@(0)}
[d(0.5)l] !{\xcaph[1]@(0)}
[d(0.5)r] !{\xcaph[1]@(0)}
[d(0.5)] !{\xcaph[1]@(0)}
[d(0.5)] !{\xcaph[1]@(0)}
[u(1)lllll] !{\xcaph[1]@(0)}
[d(0.5)l] !{\xcaph[1]@(0)}
[d(0.5)r] !{\xcaph[1]@(0)}
[d(0.5)] !{\xcaph[1]@(0)}
[u(0.5)llll] !{\xcaph[1]@(0)}
[d(0.5)l] !{\xcaph[1]@(0)}
[d(0.5)r] !{\xcaph[1]@(0)}
[lll] !{\xcaph[1]@(0)}
[d(0.5)l] !{\xcaph[1]@(0)}
[d(0.5)l] !{\xcaph[1]@(0)}
[d(0.5)l] !{\xcaph[1]@(0)}
}\hspace{-3.5pc},
\end{equation}
where arrows indicate the deletion of the generators marked in red.
To the remaining half twist, which we readily identify with $\Delta_{m-3}$, we apply the same procedure.
And we do this inductively until the remaining half twist is $\Delta_3,\Delta_2$, or $\Delta_1$, where $\Delta_1$ is just the trivial 1-strand braid.
Applying the procedure to $\Delta_3$ just yields the split union of $ a_1^2$ and one strand. On $\Delta_2= a_1^2$ and $\Delta_1$ it does not do anything.
This inductive procedure yields a braid $\beta_m$, which closes to a split union of torus links of braid index 2.
As before we illustrate this for $m=7$. %, again the generators that are removed are marked in red.
\begin{equation}\label{eq:Delta72}
\Delta_7=~\xygraph{
!{0;/r0.7pc/:}
[u(3)]!{\xcapv[6]@(0)}
[lu]!{\xcapv[6]@(0)}
[lu]!{\xcapv[6]@(0)}
[lu]!{\xcapv[6]@(0)}
[lu]!{\xcapv[6]@(0)}
[lu]!{\xcapv[6]@(0)}
[lu]!{\xcapv[6]@(0)}
[u(0.5)] !{\xcaph[1]@(0)}
[d(0.5)] !{\xcaph[1]@(0)}
[d(0.5)] !{\xcaph[1]@(0)}
[d(0.5)] !{\xcaph[1]@(0)}
[d(0.5)] !{\xcaph[1]@(0)}
[d(0.5)] !{\xcaph[1]@(0)}
[u(1.5)llllll] !{\xcaph[1]@(0)}
[d(0.5)] !{\xcaph[1]@(0)}
[d(0.5)] !{\xcaph[1]@(0)}
[d(0.5)] !{\xcaph[1]@(0)}
[d(0.5)] !{\xcaph[1]@(0)}
[u(1)lllll] !{\xcaph[1]@(0)}
[d(0.5)] !{\xcaph[1]@(0)}
[d(0.5)] !{\xcaph[1]@(0)}
[d(0.5)] !{\xcaph[1]@(0)}
[u(0.5)llll] !{\xcaph[1]@(0)}
[d(0.5)] !{\xcaph[1]@(0)}
[d(0.5)] !{\xcaph[1]@(0)}
[lll] !{\xcaph[1]@(0)}
[d(0.5)] !{\xcaph[1]@(0)}
[d(0.5)ll] !{\xcaph[1]@(0)}
}\hspace{-3.8pc}\overset{\eqref{eq:Delta71}}{\longrightarrow}
~\xygraph{
!{0;/r0.7pc/:}
[u(2.75)]!{\xcapv[5.5]@(0)}
[lu]!{\xcapv[5.5]@(0)}
[lu]!{\xcapv[5.5]@(0)}
[lu]!{\xcapv[5.5]@(0)}
[lu]!{\xcapv[5.5]@(0)}
[lu]!{\xcapv[5.5]@(0)}
[lu]!{\xcapv[5.5]@(0)}
[u(0.5)] !{\xcaph[1]@(0)}
[d(0.5)l] !{\xcaph[1]@(0)}
[d(0.5)r] !{\xcaph[1]@(0)}
[d(0.5)] !{\xcaph[1]@(0)}
[d(0.5)] !{\color{red}\xcaph[1]@(0)}
[u(1)lllll] !{\color{black}\xcaph[1]@(0)}
[d(0.5)l] !{\xcaph[1]@(0)}
[d(0.5)r] !{\xcaph[1]@(0)}
[d(0.5)] !{\xcaph[1]@(0)}
[u(0.5)llll] !{\xcaph[1]@(0)}
[d(0.5)l] !{\xcaph[1]@(0)}
[d(0.5)r] !{\xcaph[1]@(0)}
[lll] !{\xcaph[1]@(0)}
[d(0.5)l] !{\xcaph[1]@(0)}
[d(0.5)l] !{\xcaph[1]@(0)}
[d(0.5)l] !{\xcaph[1]@(0)}
}\hspace{-3.5pc}\longrightarrow\cdots\longrightarrow
~\xygraph{
!{0;/r0.7pc/:}
[u(2.75)]!{\xcapv[5.5]@(0)}
[lu]!{\xcapv[5.5]@(0)}
[lu]!{\xcapv[5.5]@(0)}
[lu]!{\xcapv[5.5]@(0)}
[lu]!{\xcapv[5.5]@(0)}
[lu]!{\xcapv[5.5]@(0)}
[lu]!{\xcapv[5.5]@(0)}
[u(0.5)] !{\xcaph[1]@(0)}
[d(0.5)l] !{\xcaph[1]@(0)}
[d(0.5)l] !{\xcaph[1]@(0)}
[d(0.5)l] !{\xcaph[1]@(0)}
[d(0.5)l] !{\xcaph[1]@(0)}
[d(0.5)l] !{\xcaph[1]@(0)}
[d(0.5)l] !{\xcaph[1]@(0)}
[d(0.5)l] !{\xcaph[1]@(0)}
[d(0.5)l] !{\xcaph[1]@(0)}
[d(0.5)l] !{\xcaph[1]@(0)}
[u(3)r] !{\xcaph[1]@(0)}
[d(0.5)l] !{\xcaph[1]@(0)}
[d(0.5)l] !{\xcaph[1]@(0)}
[d(0.5)l] !{\xcaph[1]@(0)}
}\hspace{-3.5pc}
=\beta_7
\end{equation}
The length $\l(\beta_m)$ of $\beta_m$ is described by the following formula. %, where we write $m$ as $3l$, $3l+1$, or $3l+2$, respectively.
\[
\l(\beta_m) = 2(m-2)+2(m-5)+2(m-8)+\cdots= \left\{\begin{array}{ll}
                 \hspace{-4pt}(3l-1)l &\hspace{-4pt}\text{for }m=3l\\
		\hspace{-4pt}(3l+1)l &\hspace{-4pt}\text{for } m=3l+1\\
		\hspace{-4pt}(3l+3)l+1 &\hspace{-4pt}\text{for } m=3l+2\\
                \end{array}\right. .
\]

We use the above to obtain a braid $\gamma_m$ that closes to a $T_{2,n}$ by deleting generators in $\Delta_m^2$,
which shows that $T_{2,n}$ is subword-adjacent to $T_{m,m}$.
For this we write %$\Delta_m^2$ as
\[\Delta_m^2=\Delta_m\Delta_m=( a_1 a_2\cdots a_{m-1})( a_1 a_2\cdots a_{m-2})\widetilde{\Delta_{m-2}}\Delta_m,\]
where $\widetilde{\Delta_{m-2}}$ is a half twist on the first $m-2$ strands.
Now, we apply the above deleting algorithm to $\widetilde{\Delta_{m-2}}$, which is seen as $\Delta_{m-2}$, and $\Delta_m$ yielding
\[\gamma_m=( a_1 a_2\cdots a_{m-1})( a_1 a_2\cdots a_{m-2})\widetilde{\beta_{m-2}}\beta_m,\]
where $\widetilde{\beta_{m-2}}$ is the $m$ strand braid which is obtained by having $\beta_{m-2}$ on the first $m-2$ strands. The braid $\gamma_m$ is of the form
\[\gamma_m=( a_1 a_2\cdots a_{m-1})( a_1 a_2\cdots a_{m-2}) a_1^{\alpha_1} a_3^{\alpha_3}\cdots a_{2k-1}^{\alpha_{2k-1}},\]
where $k=\lfloor\frac{m}{2}\rfloor$ and $\alpha_k$ are positive integers.
As above we illustrate this for $m=7$. %As before, the generators that are removed are marked in red.
\begin{align*}
 \Delta_7^2%&
 =~\xygraph{
!{0;/r0.6pc/:}
[u(7)]!{\xcapv[14]@(0)}
[lu]!{\xcapv[14]@(0)}
[lu]!{\xcapv[14]@(0)}
[lu]!{\xcapv[14]@(0)}
[lu]!{\xcapv[14]@(0)}
[lu]!{\xcapv[14]@(0)}
[lu]!{\xcapv[14]@(0)}
[u(0.5)] !{\xcaph[1]@(0)}
[d(0.5)] !{\xcaph[1]@(0)}
[d(0.5)] !{\xcaph[1]@(0)}
[d(0.5)] !{\xcaph[1]@(0)}
[d(0.5)] !{\xcaph[1]@(0)}
[d(0.5)] !{\xcaph[1]@(0)}
[u(1.5)llllll] !{\xcaph[1]@(0)}
[d(0.5)] !{\xcaph[1]@(0)}
[d(0.5)] !{\xcaph[1]@(0)}
[d(0.5)] !{\xcaph[1]@(0)}
[d(0.5)] !{\xcaph[1]@(0)}
[u(1)lllll] !{\xcaph[1]@(0)}
[d(0.5)] !{\xcaph[1]@(0)}
[d(0.5)] !{\xcaph[1]@(0)}
[d(0.5)] !{\xcaph[1]@(0)}
[u(0.5)llll] !{\xcaph[1]@(0)}
[d(0.5)] !{\xcaph[1]@(0)}
[d(0.5)] !{\xcaph[1]@(0)}
[lll] !{\xcaph[1]@(0)}
[d(0.5)] !{\xcaph[1]@(0)}
[d(0.5)ll] !{\xcaph[1]@(0)}
[ddl] !{\xcaph[1]@(0)}
[d(0.5)] !{\xcaph[1]@(0)}
[d(0.5)] !{\xcaph[1]@(0)}
[d(0.5)] !{\color{red}\xcaph[1]@(0)}
[d(1.5)] !{\color{black}\xcaph[1]@(0)}
[d(0.5)] !{\xcaph[1]@(0)}
[u(2.5)llllll] !{\xcaph[1]@(0)}
[d(0.5)] !{\xcaph[1]@(0)}
[d(0.5)] !{\xcaph[1]@(0)}
[d(1.5)] !{\xcaph[1]@(0)}
[d(0.5)] !{\xcaph[1]@(0)}
[u(2)lllll] !{\xcaph[1]@(0)}
[d(0.5)] !{\xcaph[1]@(0)}
[d(1.5)] !{\xcaph[1]@(0)}
[d(0.5)] !{\xcaph[1]@(0)}
[u(1.5)llll] !{\xcaph[1]@(0)}
[d(1.5)] !{\xcaph[1]@(0)}
[d(0.5)] !{\xcaph[1]@(0)}
[lll] !{\xcaph[1]@(0)}
[d(0.5)] !{\xcaph[1]@(0)}
[d(0.5)ll] !{\xcaph[1]@(0)}
}\hspace{-8.1pc}
\rightarrow
~\xygraph{
!{0;/r0.6pc/:}
[u(7)]!{\xcapv[14]@(0)}
[lu]!{\xcapv[14]@(0)}
[lu]!{\xcapv[14]@(0)}
[lu]!{\xcapv[14]@(0)}
[lu]!{\xcapv[14]@(0)}
[lu]!{\xcapv[14]@(0)}
[lu]!{\xcapv[14]@(0)}
[u(0.5)] !{\xcaph[1]@(0)}
[d(0.5)] !{\xcaph[1]@(0)}
[d(0.5)] !{\xcaph[1]@(0)}
[d(0.5)] !{\xcaph[1]@(0)}
[d(0.5)] !{\xcaph[1]@(0)}
[d(0.5)] !{\xcaph[1]@(0)}
[u(1.5)llllll] !{\xcaph[1]@(0)}
[d(0.5)] !{\xcaph[1]@(0)}
[d(0.5)] !{\xcaph[1]@(0)}
[d(0.5)] !{\xcaph[1]@(0)}
[d(0.5)] !{\xcaph[1]@(0)}
[u(1)lllll] !{\xcaph[1]@(0)}
[d(0.5)] !{\xcaph[1]@(0)}
[d(0.5)] !{\xcaph[1]@(0)}
[d(0.5)] !{\xcaph[1]@(0)}
[u(0.5)llll] !{\xcaph[1]@(0)}
[d(0.5)] !{\xcaph[1]@(0)}
[d(0.5)] !{\xcaph[1]@(0)}
[lll] !{\xcaph[1]@(0)}
[d(0.5)] !{\xcaph[1]@(0)}
[d(0.5)ll] !{\xcaph[1]@(0)}
[ddl] !{\xcaph[1]@(0)}
[d(0.5)] !{\xcaph[1]@(0)}
[d(0.5)] !{\xcaph[1]@(0)}
[d(2)r] !{\xcaph[1]@(0)}
[d(0.5)] !{\xcaph[1]@(0)}
[u(2.5)llllll] !{\xcaph[1]@(0)}
[d(0.5)] !{\xcaph[1]@(0)}
[d(0.5)] !{\xcaph[1]@(0)}
[d(1.5)] !{\xcaph[1]@(0)}
[d(0.5)] !{\xcaph[1]@(0)}
[u(2)lllll] !{\xcaph[1]@(0)}
[d(0.5)] !{\xcaph[1]@(0)}
[d(1.5)] !{\xcaph[1]@(0)}
[d(0.5)] !{\xcaph[1]@(0)}
[u(1.5)llll] !{\xcaph[1]@(0)}
[d(1.5)] !{\xcaph[1]@(0)}
[d(0.5)] !{\xcaph[1]@(0)}
[lll] !{\xcaph[1]@(0)}
[d(0.5)] !{\xcaph[1]@(0)}
[d(0.5)ll] !{\xcaph[1]@(0)}
}\hspace{-8.1pc}=
~\xygraph{
!{0;/r0.6pc/:}
[u(7.25)]!{\xcapv[14.5]@(0)}
[lu]!{\xcapv[14.5]@(0)}
[lu]!{\xcapv[14.5]@(0)}
[lu]!{\xcapv[14.5]@(0)}
[lu]!{\xcapv[14.5]@(0)}
[lu]!{\xcapv[14.5]@(0)}
[lu]!{\xcapv[14.5]@(0)}
[u(0.5)] !{\xcaph[1]@(0)}
[d(0.5)] !{\xcaph[1]@(0)}
[d(0.5)] !{\xcaph[1]@(0)}
[d(0.5)] !{\xcaph[1]@(0)}
[d(0.5)] !{\xcaph[1]@(0)}
[d(0.5)] !{\xcaph[1]@(0)}
[u(1.5)llllll] !{\xcaph[1]@(0)}
[d(0.5)] !{\xcaph[1]@(0)}
[d(0.5)] !{\xcaph[1]@(0)}
[d(0.5)] !{\xcaph[1]@(0)}
[d(0.5)] !{\xcaph[1]@(0)}
[u(1)lllll] !{\xcaph[1]@(0)}
[d(0.5)] !{\xcaph[1]@(0)}
[d(0.5)] !{\xcaph[1]@(0)}
[d(0.5)] !{\xcaph[1]@(0)}
[u(0.5)llll] !{\xcaph[1]@(0)}
[d(0.5)] !{\xcaph[1]@(0)}
[d(0.5)] !{\xcaph[1]@(0)}
[lll] !{\xcaph[1]@(0)}
[d(0.5)] !{\xcaph[1]@(0)}
[d(0.5)ll] !{\xcaph[1]@(0)}
[d(1.5)l] !{\xcaph[1]@(0)}
[d(0.5)l] !{\xcaph[1]@(0)}
[d(0.5)] !{\color{red}\xcaph[1]@(0)}
[d(0.5)] !{\color{black}\xcaph[1]@(0)}
[d(2)r] !{\xcaph[1]@(0)}
[d(0.5)] !{\xcaph[1]@(0)}
[u(2.5)llllll] !{\xcaph[1]@(0)}
[d(0.5)l] !{\xcaph[1]@(0)}
[d(0.5)] !{\color{red}\xcaph[1]@(0)}
[d(1.5)r] !{\color{black}\xcaph[1]@(0)}
[d(0.5)] !{\xcaph[1]@(0)}
[u(1.5)lllll] !{\xcaph[1]@(0)}
[d(2)r] !{\xcaph[1]@(0)}
[d(0.5)] !{\xcaph[1]@(0)}
[u(2)llll] !{\xcaph[1]@(0)}
[d(2)] !{\xcaph[1]@(0)}
[d(0.5)] !{\xcaph[1]@(0)}
[lll] !{\xcaph[1]@(0)}
[d(0.5)] !{\xcaph[1]@(0)}
[d(0.5)ll] !{\xcaph[1]@(0)}
}\hspace{-8.4pc}
%\\&
\rightarrow
~\xygraph{
!{0;/r0.6pc/:}
[u(6.75)]!{\xcapv[13.5]@(0)}
[lu]!{\xcapv[13.5]@(0)}
[lu]!{\xcapv[13.5]@(0)}
[lu]!{\xcapv[13.5]@(0)}
[lu]!{\xcapv[13.5]@(0)}
[lu]!{\xcapv[13.5]@(0)}
[lu]!{\xcapv[13.5]@(0)}
[u(0.5)] !{\xcaph[1]@(0)}
[d(0.5)] !{\xcaph[1]@(0)}
[d(0.5)] !{\xcaph[1]@(0)}
[d(0.5)] !{\xcaph[1]@(0)}
[d(0.5)] !{\xcaph[1]@(0)}
[d(0.5)] !{\xcaph[1]@(0)}
[u(1.5)llllll] !{\xcaph[1]@(0)}
[d(0.5)] !{\xcaph[1]@(0)}
[d(0.5)] !{\xcaph[1]@(0)}
[d(0.5)] !{\xcaph[1]@(0)}
[d(0.5)] !{\xcaph[1]@(0)}
[u(1)lllll] !{\xcaph[1]@(0)}
[d(0.5)] !{\xcaph[1]@(0)}
[d(0.5)] !{\xcaph[1]@(0)}
[d(0.5)] !{\xcaph[1]@(0)}
[u(0.5)llll] !{\xcaph[1]@(0)}
[d(0.5)] !{\xcaph[1]@(0)}
[d(0.5)] !{\xcaph[1]@(0)}
[lll] !{\xcaph[1]@(0)}
[d(0.5)] !{\xcaph[1]@(0)}
[d(0.5)ll] !{\xcaph[1]@(0)}
[d(1.5)l] !{\xcaph[1]@(0)}
[d(0.5)l] !{\xcaph[1]@(0)}
[d(0.5)r] !{\xcaph[1]@(0)}
[d(2)r] !{\xcaph[1]@(0)}
[d(0.5)] !{\xcaph[1]@(0)}
[u(2.5)llllll] !{\xcaph[1]@(0)}
[d(0.5)l] !{\xcaph[1]@(0)}
[d(2)rr] !{\xcaph[1]@(0)}
[d(0.5)] !{\xcaph[1]@(0)}
[u(2)lllll] !{\xcaph[1]@(0)}
[d(2)r] !{\xcaph[1]@(0)}
[d(0.5)] !{\xcaph[1]@(0)}
[u(2)llll] !{\xcaph[1]@(0)}
[d(2)] !{\xcaph[1]@(0)}
[d(0.5)] !{\xcaph[1]@(0)}
[lll] !{\xcaph[1]@(0)}
[d(0.5)] !{\xcaph[1]@(0)}
[d(0.5)ll] !{\xcaph[1]@(0)}
}\hspace{-7.9pc}
\overset{\eqref{eq:Delta72}}{\rightarrow}
~\hspace{-0.2pc}\xygraph{
!{0;/r0.6pc/:}
[u(6.75)]!{\xcapv[13.5]@(0)}
[lu]!{\xcapv[13.5]@(0)}
[lu]!{\xcapv[13.5]@(0)}
[lu]!{\xcapv[13.5]@(0)}
[lu]!{\xcapv[13.5]@(0)}
[lu]!{\xcapv[13.5]@(0)}
[lu]!{\xcapv[13.5]@(0)}
[u(0.5)] !{\xcaph[1]@(0)}
[d(0.5)l] !{\xcaph[1]@(0)}
[d(0.5)l] !{\xcaph[1]@(0)}
[d(0.5)l] !{\xcaph[1]@(0)}
[d(0.5)l] !{\xcaph[1]@(0)}
[d(0.5)l] !{\xcaph[1]@(0)}
[d(0.5)l] !{\xcaph[1]@(0)}
[d(0.5)l] !{\xcaph[1]@(0)}
[d(0.5)l] !{\xcaph[1]@(0)}
[d(0.5)l] !{\xcaph[1]@(0)}
[u(3)r] !{\xcaph[1]@(0)}
[d(0.5)l] !{\xcaph[1]@(0)}
[d(0.5)l] !{\xcaph[1]@(0)}
[d(0.5)l] !{\xcaph[1]@(0)}
[d(3.5)lll] !{\xcaph[1]@(0)}
[d(0.5)l] !{\xcaph[1]@(0)}
[d(0.5)r] !{\xcaph[1]@(0)}
[d(2)r] !{\xcaph[1]@(0)}
[d(0.5)] !{\xcaph[1]@(0)}
[u(2.5)llllll] !{\xcaph[1]@(0)}
[d(0.5)l] !{\xcaph[1]@(0)}
[d(2)rr] !{\xcaph[1]@(0)}
[d(0.5)] !{\xcaph[1]@(0)}
[u(2)lllll] !{\xcaph[1]@(0)}
[d(2)r] !{\xcaph[1]@(0)}
[d(0.5)] !{\xcaph[1]@(0)}
[u(2)llll] !{\xcaph[1]@(0)}
[d(2)] !{\xcaph[1]@(0)}
[d(0.5)] !{\xcaph[1]@(0)}
[lll] !{\xcaph[1]@(0)}
[d(0.5)] !{\xcaph[1]@(0)}
[d(0.5)ll] !{\xcaph[1]@(0)}
}\hspace{-8pc}=\gamma_7
\end{align*}
The closure of $\gamma_m$ is a braid index two torus link $T_{2,n}$.
This follows from observing that the closure of $( a_1 a_2\cdots a_{m-1})( a_1 a_2\cdots a_{m-2})$
is $T_{2,m-1}$.
Since \[\l(\gamma_m)-\l(( a_1 a_2\cdots a_{m-1})( a_1 a_2\cdots a_{m-2}))=\l(\beta_m)+\l(\beta_{m-2})\] we see that $n=m-1+\l(\beta_{m-2})+\l(\beta_{m})$;
i.e.\  the closure of $\gamma_m$ is %$T_{2,n}=%T(2,\l(\gamma_m)-m+2)=
$T_{2,m-1+\l(\beta_{m-2})+\l(\beta_{m})}$.
%Since $\l(\gamma_m)=2m-3+\l(\beta_{m-2})+\l(\beta_{m})$
%we conclude that the closure of $\gamma_m$ is $T(2,n)$ with $n=m-1+\l(\beta_{m-2})+\l(\beta_{m})$,
With the above calculations for $\l(\beta_m)$ we get
\begin{align*}
n &= 3l-1+(3l-2)(l-1)+(3l-1)l = 6l^2-3l+1,\\
n &= (3l+1-1)+(3l)(l-1)+1+(3l+1)l = 6l^2+l+1, \\
n &= (3l+2-1)+(3l-1)l+(3l+3)l+1=6l^2+5l+2,
\end{align*}
for $m=3l$, $m=3l+1$, and $m=3l+2$, respectively.
This finishes the proof since $n$ is the largest integer with $n\leq \frac{2m^2+4}{3}-m$.
\end{proof}

\section{Calculation of $\Upsilon$ for torus knots of small braid index}\label{sec:upsilonfortorusknots}
For completeness, we provide the calculations that yield the $\upsilon$-values given in%equation
~\eqref{upsilonsfortk}.
\begin{prop}\label{prop:upsilonforT3T4}
For positive integers $n$, we have
\[\upsilon (T_{3,3n+1})=-2n,\quad \upsilon (T_{3,3n+2})=-2n-1,\et \upsilon (T_{4,2n+1})=-2n.\]
\end{prop}
\begin{Remark}
More generally, the calculation we provide below in the proof of Proposition~\ref{prop:upsilonforT3T4} allows to determine the function $\Upsilon_T:[0,2]\to\R$ for torus knots $T$ of braid index 3 and 4:
%For this, recall that $\Upsilon_T$ is continuous, piecewise linear, symmetric (i.e.~$\Upsilon_T(t)=(2-t)$), and has slope $\tau(T)$ at $\Upsilon_T(0)=0$,
    %by general properties of $\Upsilon$~\cite{OSS_2014}. Therefore, Proposition~\ref{prop:upsilonforT3T4} and the following describes $\Upsilon(T)$:
For all positive integers $n$, we have
\begin{align*}\Upsilon_{T_{3,3n+1}}(t)&=
\left\{\begin{aligned}-3nt&\text{ for } 0\leq t\leq \frac{2}{3}\\
-2n&\text{ for } \frac{2}{3}\leq t\leq 1\end{aligned}\right.,\\
\Upsilon_{T_{3,3n+2}}(t)&=
\left\{\begin{aligned}-(3n+1)t&\text{ for } 0\leq t\leq \frac{2}{3}\\
-2n-t&\text{ for } \frac{2}{3}\leq t\leq 1\end{aligned}\right.,\\
\Upsilon_{T_{4,4n+1}}(t)&=\left\{\begin{aligned}-6nt&\text{ for } 0\leq t\leq \frac{1}{2}\\
-2n-2nt&\text{ for } \frac{1}{2}\leq t\leq 1\end{aligned}\right.,\\\et
\Upsilon_{T_{4,4n+3}}(t)&=
\left\{\begin{aligned}-(6n+3)t&\text{ for } 0\leq t\leq \frac{1}{2}\\
-2n-(2n+3)t&\text{ for } \frac{1}{2}\leq t\leq \frac{2}{3}\\
-2n-2-2nt&\text{ for } \frac{2}{3}\leq t\leq 1\end{aligned}\right..
\end{align*}
This fully describes $\Upsilon_T$ since it is symmetric, i.e.~$\Upsilon_T(t)=(2-t)$; see~\cite{OSS_2014}.
\end{Remark}
Our calculations have convinced us that, for a general torus knot $T_{p,q}$,
$\Upsilon_{T_{p,q}}(t)$ might look similar to the homogenization of the signature profile of torus knots;
i.e.~the following function, studied by Gambaudo and Ghys~\cite{GambaudoGhys_BraidsSignatures}:
\[Sign_{T_{p,q}}\colon[0,2]\to\R,\quad t\mapsto
\lim_{k\to\infty}\frac{\s_{e^{\pi it}}\left(\overline{((a_1\cdots a_{p-1})^q)^k}\right)}{k},\]
where $\so$ denotes the Levine-Tristram signature. We hope to explore this further in the future.
%All we do here is purely combinatorial.
The only Heegaard-Floer theory input in the proof of Proposition~\ref{prop:upsilonforT3T4} is the following combinatorial procedure to determine $\Upsilon$ for torus knots (or more generally $L$-space knots)~\cite{OSS_2014}:

Write the Alexander polynomial
\[\Delta (T_{p,q})=t^{-\frac{(p-1)(q-1)}{2}}\frac{(t^{pq}-1)(t-1)}{(t^{p}-1)(t^{q}-1)}\]
as $\sum_{k=0}^{l}(-1)^kt^{\alpha_k}$, where $(\alpha_k)_{k=0}^l$ is a decreasing sequence of integers. Construct a corresponding decreasing sequence of integers $(m_k)_{k=0}^l$ defined by
\begin{equation}\label{eq:defmk}
m_0=0,
m_{2k}=m_{2k-1}-1,\et
m_{2k+1}=m_{2k}-2(\alpha_{2k}-\alpha_{2k+1})+1.\\
\end{equation}
Then one has
\begin{equation}
\Upsilon_{T_{p,q}}(t)=\max_{0\leq 2k\leq l}\{m_{2k}-t\alpha_{2k}\}\;\;\text{\cite[Theorem~1.15]{OSS_2014}}.
\end{equation}
In particular,
\[\Upsilon_{T_{p,q}}(1)=\upsilon(T_{p,q})
=\max_{0\leq 2k\leq l}\{m_{2k}-\alpha_{2k}\}.
%=\{m_{2\lfloor{frac{n-1}{2}\rfloor}}-\alpha_{2\lfloor{frac{n}{2}\rfloor}}
\]
In fact, for the calculation one only needs the evenly indexed $m_k$, for which~\eqref{eq:defmk} can be shortened to
\begin{equation}\label{eq:m_i}
m_0=0\quad
m_{2k}=m_{2k-2}-2(\alpha_{2k-2}-\alpha_{2k-1}).
\end{equation}

\begin{proof}[Proof of Proposition~\ref{prop:upsilonforT3T4}]
We observe that
\begin{align*}
\Delta (T_{3,3n+1})
&=t^{-3n}\frac{(t^{9n+3}-1)(t-1)}{(t^{3n+1}-1)(t^{3}-1)}\\
&=t^{-3n}\frac{t^{9n+2}+t^{9n+1}+\cdots+1}{(t^{3n}+t^{3n-1}+\cdots+1)(t^{2}+t+1)}\\
&=\frac{t^{6n+2}+t^{6n+1}+\cdots+t^{-3n+1}+t^{-3n}}{t^{3n+2}+2t^{3n+1}+3t^{3n}+3t^{3n-1}+\cdots+3t^3+3t^2+2t+1}\\
&=(t^{3n}-t^{3n-1})+(t^{3n-3}-t^{3n-4})+\cdots+(t^{3}-t^{2})+1-t^{-2}+\cdots\\ %-t^{-2}-t^{-3}+\cdots+t^{-3n+1}-t^{-3n}\\
&=\sum_{i=0}^{n-1}(t^{3n-3i}-t^{3n-3i-1})+1+\sum_{i=1}^{n}(- t^{-3i+1}+t^{-3i}).
\end{align*}
%In terms of $l$, which denotes the number of monomials in $\Delta(T_{p,q})$, and $\alpha_k$ this means
In other words, $\Delta(T_{p,q})=\sum_{k=0}^{l}(-1)^kt^{\alpha_k}$ for
\[l=4n,\quad \alpha_{2k}=3n-3k, \et
\alpha_{2k-2}-\alpha_{2k-1}=
\left\{\begin{aligned}1&\text{ for }k\leq n\\ 2&\text{ for }k> n\end{aligned}\right..\]
Therefore, \eqref{eq:m_i} yields
$m_{2k}=-2k$%\text
{ for }$k\leq n$ %\et
and $m_{2k}=2n-4k$%\text
{ for }$k\geq n,$
and so
\[\upsilon(T_{3,3n+1})
=\max_{0\leq 2k\leq l}\{m_{2k}-\alpha_{2k}\}=m_{2n}-\alpha_{2n}=-2n.\]

Similarly, one calculates
\begin{align*}
\Delta (T_{3,3n+2})
%&=t^{-3n-1}\frac{(t^{9n+6}-1)(t-1)}{(t^{3n+2}-1)(t^{3}-1)}\\
%&=t^{-3n-1}\frac{t^{9n+5}+t^{9n+4}+\cdots+1}{(t^{3n+1}+t^{3n}+\cdots+1)(t^{2}+t+1)}\\
&=\frac{t^{6n+4}+t^{6n+3}+\cdots+t^{-3n}+t^{-3n-1}}{t^{3n+3}+2t^{3n+2}+3t^{3n+1}+3t^{3n}+\cdots+3t^3+3t^2+2t+1}\\
&=t^{3n+1}-t^{3n}+t^{3n-2}-t^{3n-3}+\cdots+t^{4}-t^{3}+t-1+\cdots\\ %-t^{-2}-t^{-3}+\cdots+t^{-3n+1}-t^{-3n}\\
&=\sum_{i=0}^{n-1}(t^{3n-3i+1}-t^{3n-3i})+t-1+t^{-1}+\sum_{i=1}^{n}( -t^{-3i}+t^{-3i-1}),
\end{align*}
which means that $\Delta(T_{p,q})=\sum_{k=0}^{l}(-1)^kt^{\alpha_k}$ for $l=4n+2$,
\[\quad \alpha_{2k}=
\left\{\begin{aligned}3n+1-3k&\text{ for } k\leq n\\
3n+2-3k&\text{ for } k> n\end{aligned}\right.,\et\alpha_{2k-2}-\alpha_{2k-1}=
\left\{\begin{aligned}1&\text{ for } k\leq n+1\\
2&\text{ for } k> n+1\end{aligned}\right..\]
Thus, $m_{2k}=-2k$%\text
{ for }$k\leq n+1$ %\et
and $m_{2k}=2n+2-4k$%\text
 { for }$k\geq n+1$, which yields
\[\upsilon(T_{3,3n+2})
=\max_{0\leq 2k\leq l}\{m_{2k}-\alpha_{2k}\}=m_{2n}-\alpha_{2n}=m_{2n+2}-\alpha_{2n+2}=-2n-1.\]

For $T_{4,2n+1}$, we calculate first when $n$ is even, i.e.~$n=2s$ for a positive integer $s$: %One has
\begin{align*}
&\Delta (T_{4,4{s}+1}) = t^{-6{s}}\frac{(t^{16{s}+4}-1)(t-1)}{(t^{4{s}+1}-1)(t^{4}-1)}
= \frac{t^{-6{s}}(t^{16{s}+3}+t^{16{s}+2}+\cdots+1)}{(t^{4{s}}+t^{4{s}-1}+\cdots+1)(t^3+t^{2}+t+1)}\\
&=\frac{t^{10{s}+3}+t^{10{s}+2}+\cdots+t^{-6{s}+1}+t^{-6{s}}}
{t^{4{s}+3}+2t^{4{s}+2}+3t^{4{s}+1}+4t^{4{s}}+4t^{4{s}-1}+\cdots+4t^4+4t^3+3t^2+2t+1}\\
%&= t^{6{s}}-t^{6{s}-1}+t^{6{s}-4}-t^{6{s}-5}+\cdots+t^{2{s}}-t^{2{s}-1}+t^{2{s}-4}-t^{2{s}-5}+\cdots\\
%&\hspace{176.5pt}+t^{2{s}-1}-t^{2{s}-2}+t^{2{s}-5}-t^{2{s}-6}+\cdots\\ %-t^{-2}-t^{-3}+\cdots+t^{-3{s}+1}-t^{-3{s}}\\
&= \sum_{i=0}^{{s}-1}(t^{6{s}-4i}-t^{6{s}-4i-1})
+\sum_{i={s}}^{2{s}-1}(t^{6{s}-4i}-t^{6{s}-4i-2})\\
&+\sum_{i=2{s}}^{3{s}-1}(t^{6{s}-4i}-t^{6{s}-4i-3})+t^{-6{s}},
\end{align*}
which means
\[l=3n,\quad \alpha_{2k}=3n-4k,
\et\alpha_{2k-2}-\alpha_{2k-1}=
\left\{\begin{aligned}1&\text{ for } k\leq \frac{n}{2}\\
2&\text{ for } \frac{n}{2}< k\leq n\\
3&\text{ for } n< k%\leq \frac{3n}{2}
\end{aligned}\right.\] Therefore, we have
\[m_{2k}=\left\{\begin{aligned} -2k&\text{ for }k \leq \frac{n}{2}\\
 n-4k&\text{ for }\frac{n}{2}\leq k\leq n\\
3n-6k&\text{ for } n \leq k %\leq \frac{3n}{2}
\end{aligned}\right.,\]which yields that $\upsilon(T_{4,2n+1})$ equals
\[\max_{0\leq 2k\leq l}\{m_{2k}-\alpha_{2k}\}=m_{n}-\alpha_{n}=m_{n+2}-\alpha_{n+2}=\cdots=m_{2n}-\alpha_{2n}=-2n.\]
Finally, for $n$ odd, a similar calculation yields
\begin{align*}
%\quad \quad&t^{-3n}\frac{(t^{8n+4}-1)(t-1)}{(t^{2n+1}-1)(t^{4}-1)}\\
%=&t^{-3n}\frac{t^{8n+3}+t^{16n+2}+\cdots+1}{(t^{2n}+t^{2n-1}+\cdots+1)(t^3+t^{2}+t+1)}\\
%=&\frac{t^{5n+3}+t^{5n+2}+\cdots+t^{-3n}+t^{-3n-1}}{t^{2n+3}+2t^{2n+2}+3t^{2n+1}+4t^{2n}+4t^{2n-1}+\cdots+4t^4+4t^3+3t^2+2t+1}\\
%=&\begin{aligned}t^{3n}-t^{3n-1}+t^{3n-4}-t^{3n-5}+\cdots+t^{3n-4i}-t^{3n-4i-1}+\cdots\\
%&t^{n-1}-t^{n-2}+t^{n-5}-t^{n-6}+\cdots\end{aligned}\\ %-t^{-2}-t^{-3}+\cdots+t^{-3n+1}-t^{-3n}\\
\Delta (T_{4,2n+1})=&\sum_{i=0}^{\frac{n-1}{2}}(t^{3n-4i}-t^{3n-1-4i})+
\sum_{i=\frac{n+1}{2}}^{n}(t^{3n+1-4i}-t^{3n-1-4i})\\
&+\sum_{i=n+1}^{\frac{3n-1}{2}}(t^{3n+2-4i}-t^{3n-1-4i})+t^{-3n},
\end{align*}
which means $l=3n+1$,
\[\alpha_{2k}=
\left\{\begin{aligned} 3n-4k &\text{ for } k\leq \frac{n-1}{2}\\
3n+1-4k &\text{ for } \frac{n-1}{2}< k\leq n \\
3n+2-4k &\text{ for } n< k%\leq \frac{3n+1}{2}
\\
\end{aligned}\right.,\]\[\et \alpha_{2k-2}-\alpha_{2k-1}=
\left\{\begin{aligned}1&\text{ for } k\leq \frac{n+1}{2}\\
2&\text{ for } \frac{n+1}{2}< k\leq n+1\\
3&\text{ for } n+1< k%\leq \frac{3n+3}{2}
\end{aligned}\right..\] Therefore, we have
\[m_{2k}=\left\{\begin{aligned} -2k &\text{ for }k\leq \frac{n+1}{2}\\
 n+1-4k&\text{ for }\frac{n+1}{2}< k\leq n+1\\
3n+3-6k&\text{ for }n+1< k\end{aligned}\right..\]
This yields that $\upsilon(T_{4,2n+1})$ equals
\[\max_{0\leq 2k\leq l}\{m_{2k}-\alpha_{2k}\}
=m_{n+1}-\alpha_{n+1}=m_{n+3}-\alpha_{n+3}=\cdots=m_{2n}-\alpha_{2n}=-2n.\]

\end{proof}

%\bibliographystyle{alpha}
%\pagebreak
%\bibliography{peterbib}
%\def\cprime{$'$} \def\cprime{$'$}
%\begin{thebibliography}{GZN00}
\def\cprime{$'$}

\end{document}

%% file: vertices.pdf_tex
%% Creator: Inkscape 0.48.5, www.inkscape.org
%% PDF/EPS/PS + LaTeX output extension by Johan Engelen, 2010
%% Accompanies image file 'vertices.pdf' (pdf, eps, ps)
%%
%% To include the image in your LaTeX document, write
%%   \input{<filename>.pdf_tex}
%%  instead of
%%   \includegraphics{<filename>.pdf}
%% To scale the image, write
%%   \def\svgwidth{<desired width>}
%%   \input{<filename>.pdf_tex}
%%  instead of
%%   \includegraphics[width=<desired width>]{<filename>.pdf}
%%
%% Images with a different path to the parent latex file can
%% be accessed with the `import' package (which may need to be
%% installed) using
%%   \usepackage{import}
%% in the preamble, and then including the image with
%%   \import{<path to file>}{<filename>.pdf_tex}
%% Alternatively, one can specify
%%   \graphicspath{{<path to file>/}}
%% 
%% For more information, please see info/svg-inkscape on CTAN:
%%   http://tug.ctan.org/tex-archive/info/svg-inkscape
%%
\begingroup%
  \makeatletter%
  \providecommand\color[2][]{%
    \errmessage{(Inkscape) Color is used for the text in Inkscape, but the package 'color.sty' is not loaded}%
    \renewcommand\color[2][]{}%
  }%
  \providecommand\transparent[1]{%
    \errmessage{(Inkscape) Transparency is used (non-zero) for the text in Inkscape, but the package 'transparent.sty' is not loaded}%
    \renewcommand\transparent[1]{}%
  }%
  \providecommand\rotatebox[2]{#2}%
  \ifx\svgwidth\undefined%
    \setlength{\unitlength}{179.14400383bp}%
    \ifx\svgscale\undefined%
      \relax%
    \else%
      \setlength{\unitlength}{\unitlength * \real{\svgscale}}%
    \fi%
  \else%
    \setlength{\unitlength}{\svgwidth}%
  \fi%
  \global\let\svgwidth\undefined%
  \global\let\svgscale\undefined%
  \makeatother%
  \begin{picture}(1,0.28734742)%
    \put(0,0){\includegraphics[width=\unitlength]{vertices.pdf}}%
    \put(0.51341094,0.22088958){\color[rgb]{0,0,0}\makebox(0,0)[lb]{\smash{$i$}}}%
    \put(0.5157427,0.10471372){\color[rgb]{0,0,0}\makebox(0,0)[lb]{\smash{$j$}}}%
    \put(0.88764511,0.22233592){\color[rgb]{0,0,0}\makebox(0,0)[lb]{\smash{$i$}}}%
    \put(0.89607506,0.09694382){\color[rgb]{0,0,0}\makebox(0,0)[lb]{\smash{$i$}}}%
    \put(0.93610563,0.14534145){\color[rgb]{0,0,0}\makebox(0,0)[lb]{\smash{$j$}}}%
    \put(0.40063545,0.01238657){\color[rgb]{0,0,0}\makebox(0,0)[lb]{\smash{$|i-j|\geq2$}}}%
    \put(0.77781037,0.01145641){\color[rgb]{0,0,0}\makebox(0,0)[lb]{\smash{$|i-j|=1$}}}%
    \put(0.74860223,0.14609399){\color[rgb]{0,0,0}\makebox(0,0)[lb]{\smash{$i$}}}%
    \put(0.80290347,0.22674543){\color[rgb]{0,0,0}\makebox(0,0)[lb]{\smash{$j$}}}%
    \put(0.79162595,0.09880416){\color[rgb]{0,0,0}\makebox(0,0)[lb]{\smash{$j$}}}%
    \put(0.43252011,0.22463817){\color[rgb]{0,0,0}\makebox(0,0)[lb]{\smash{$j$}}}%
    \put(0.42454176,0.10438405){\color[rgb]{0,0,0}\makebox(0,0)[lb]{\smash{$i$}}}%
    \put(0.14883421,0.18060164){\color[rgb]{0,0,0}\makebox(0,0)[lb]{\smash{$i$}}}%
  \end{picture}%
\endgroup%

%% file: conjugaterealization.pdf_tex
%% Creator: Inkscape 0.48.5, www.inkscape.org
%% PDF/EPS/PS + LaTeX output extension by Johan Engelen, 2010
%% Accompanies image file 'conjugaterealization.pdf' (pdf, eps, ps)
%%
%% To include the image in your LaTeX document, write
%%   \input{<filename>.pdf_tex}
%%  instead of
%%   \includegraphics{<filename>.pdf}
%% To scale the image, write
%%   \def\svgwidth{<desired width>}
%%   \input{<filename>.pdf_tex}
%%  instead of
%%   \includegraphics[width=<desired width>]{<filename>.pdf}
%%
%% Images with a different path to the parent latex file can
%% be accessed with the `import' package (which may need to be
%% installed) using
%%   \usepackage{import}
%% in the preamble, and then including the image with
%%   \import{<path to file>}{<filename>.pdf_tex}
%% Alternatively, one can specify
%%   \graphicspath{{<path to file>/}}
%% 
%% For more information, please see info/svg-inkscape on CTAN:
%%   http://tug.ctan.org/tex-archive/info/svg-inkscape
%%
\begingroup%
  \makeatletter%
  \providecommand\color[2][]{%
    \errmessage{(Inkscape) Color is used for the text in Inkscape, but the package 'color.sty' is not loaded}%
    \renewcommand\color[2][]{}%
  }%
  \providecommand\transparent[1]{%
    \errmessage{(Inkscape) Transparency is used (non-zero) for the text in Inkscape, but the package 'transparent.sty' is not loaded}%
    \renewcommand\transparent[1]{}%
  }%
  \providecommand\rotatebox[2]{#2}%
  \ifx\svgwidth\undefined%
    \setlength{\unitlength}{218.24895914bp}%
    \ifx\svgscale\undefined%
      \relax%
    \else%
      \setlength{\unitlength}{\unitlength * \real{\svgscale}}%
    \fi%
  \else%
    \setlength{\unitlength}{\svgwidth}%
  \fi%
  \global\let\svgwidth\undefined%
  \global\let\svgscale\undefined%
  \makeatother%
  \begin{picture}(1,0.36723475)%
    \put(0,0){\includegraphics[width=\unitlength]{conjugaterealization.pdf}}%
    \put(0.3597044,0.26109201){\color[rgb]{0,0,0}\makebox(0,0)[lb]{\smash{4}}}%
    \put(0.56134718,0.07190207){\color[rgb]{0,0,0}\makebox(0,0)[lb]{\smash{1}}}%
    \put(0.5160815,0.15266157){\color[rgb]{0,0,0}\makebox(0,0)[lb]{\smash{2}}}%
    \put(0.48583778,0.19809971){\color[rgb]{0,0,0}\makebox(0,0)[lb]{\smash{3}}}%
    \put(0.67682406,0.0314565){\color[rgb]{0,0,0}\makebox(0,0)[lb]{\smash{$\partial D=S^1$}}}%
  \end{picture}%
\endgroup%

%% file: verticesCC.pdf_tex
%% Creator: Inkscape 0.48.5, www.inkscape.org
%% PDF/EPS/PS + LaTeX output extension by Johan Engelen, 2010
%% Accompanies image file 'verticesCC.pdf' (pdf, eps, ps)
%%
%% To include the image in your LaTeX document, write
%%   \input{<filename>.pdf_tex}
%%  instead of
%%   \includegraphics{<filename>.pdf}
%% To scale the image, write
%%   \def\svgwidth{<desired width>}
%%   \input{<filename>.pdf_tex}
%%  instead of
%%   \includegraphics[width=<desired width>]{<filename>.pdf}
%%
%% Images with a different path to the parent latex file can
%% be accessed with the `import' package (which may need to be
%% installed) using
%%   \usepackage{import}
%% in the preamble, and then including the image with
%%   \import{<path to file>}{<filename>.pdf_tex}
%% Alternatively, one can specify
%%   \graphicspath{{<path to file>/}}
%% 
%% For more information, please see info/svg-inkscape on CTAN:
%%   http://tug.ctan.org/tex-archive/info/svg-inkscape
%%
\begingroup%
  \makeatletter%
  \providecommand\color[2][]{%
    \errmessage{(Inkscape) Color is used for the text in Inkscape, but the package 'color.sty' is not loaded}%
    \renewcommand\color[2][]{}%
  }%
  \providecommand\transparent[1]{%
    \errmessage{(Inkscape) Transparency is used (non-zero) for the text in Inkscape, but the package 'transparent.sty' is not loaded}%
    \renewcommand\transparent[1]{}%
  }%
  \providecommand\rotatebox[2]{#2}%
  \ifx\svgwidth\undefined%
    \setlength{\unitlength}{179.14400383bp}%
    \ifx\svgscale\undefined%
      \relax%
    \else%
      \setlength{\unitlength}{\unitlength * \real{\svgscale}}%
    \fi%
  \else%
    \setlength{\unitlength}{\svgwidth}%
  \fi%
  \global\let\svgwidth\undefined%
  \global\let\svgscale\undefined%
  \makeatother%
  \begin{picture}(1,0.28734742)%
    \put(0,0){\includegraphics[width=\unitlength]{verticesCC.pdf}}%
    \put(0.51341094,0.22088958){\color[rgb]{0,0,0}\makebox(0,0)[lb]{\smash{$i$}}}%
    \put(0.5157427,0.10471372){\color[rgb]{0,0,0}\makebox(0,0)[lb]{\smash{$j$}}}%
    \put(0.88764511,0.22233592){\color[rgb]{0,0,0}\makebox(0,0)[lb]{\smash{$i$}}}%
    \put(0.89607506,0.09694382){\color[rgb]{0,0,0}\makebox(0,0)[lb]{\smash{$i$}}}%
    \put(0.93610563,0.14534145){\color[rgb]{0,0,0}\makebox(0,0)[lb]{\smash{$j$}}}%
    \put(0.40063545,0.01238657){\color[rgb]{0,0,0}\makebox(0,0)[lb]{\smash{$|i-j|\geq2$}}}%
    \put(0.77781037,0.01145641){\color[rgb]{0,0,0}\makebox(0,0)[lb]{\smash{$|i-j|=1$}}}%
    \put(0.74860223,0.14609399){\color[rgb]{0,0,0}\makebox(0,0)[lb]{\smash{$i$}}}%
    \put(0.80290347,0.22674543){\color[rgb]{0,0,0}\makebox(0,0)[lb]{\smash{$j$}}}%
    \put(0.79162595,0.09880416){\color[rgb]{0,0,0}\makebox(0,0)[lb]{\smash{$j$}}}%
    \put(0.43252011,0.22463817){\color[rgb]{0,0,0}\makebox(0,0)[lb]{\smash{$j$}}}%
    \put(0.42454176,0.10438405){\color[rgb]{0,0,0}\makebox(0,0)[lb]{\smash{$i$}}}%
    \put(0.14883421,0.18060164){\color[rgb]{0,0,0}\makebox(0,0)[lb]{\smash{$i$}}}%
  \end{picture}%
\endgroup%

%% file: genericoutward.pdf_tex
%% Creator: Inkscape 0.48.5, www.inkscape.org
%% PDF/EPS/PS + LaTeX output extension by Johan Engelen, 2010
%% Accompanies image file 'genericoutward.pdf' (pdf, eps, ps)
%%
%% To include the image in your LaTeX document, write
%%   \input{<filename>.pdf_tex}
%%  instead of
%%   \includegraphics{<filename>.pdf}
%% To scale the image, write
%%   \def\svgwidth{<desired width>}
%%   \input{<filename>.pdf_tex}
%%  instead of
%%   \includegraphics[width=<desired width>]{<filename>.pdf}
%%
%% Images with a different path to the parent latex file can
%% be accessed with the `import' package (which may need to be
%% installed) using
%%   \usepackage{import}
%% in the preamble, and then including the image with
%%   \import{<path to file>}{<filename>.pdf_tex}
%% Alternatively, one can specify
%%   \graphicspath{{<path to file>/}}
%% 
%% For more information, please see info/svg-inkscape on CTAN:
%%   http://tug.ctan.org/tex-archive/info/svg-inkscape
%%
\begingroup%
  \makeatletter%
  \providecommand\color[2][]{%
    \errmessage{(Inkscape) Color is used for the text in Inkscape, but the package 'color.sty' is not loaded}%
    \renewcommand\color[2][]{}%
  }%
  \providecommand\transparent[1]{%
    \errmessage{(Inkscape) Transparency is used (non-zero) for the text in Inkscape, but the package 'transparent.sty' is not loaded}%
    \renewcommand\transparent[1]{}%
  }%
  \providecommand\rotatebox[2]{#2}%
  \ifx\svgwidth\undefined%
    \setlength{\unitlength}{448.62684bp}%
    \ifx\svgscale\undefined%
      \relax%
    \else%
      \setlength{\unitlength}{\unitlength * \real{\svgscale}}%
    \fi%
  \else%
    \setlength{\unitlength}{\svgwidth}%
  \fi%
  \global\let\svgwidth\undefined%
  \global\let\svgscale\undefined%
  \makeatother%
  \begin{picture}(1,0.48557561)%
    \put(0,0){\includegraphics[width=\unitlength]{genericoutward.pdf}}%
    \put(0.12299411,0.23859956){\color[rgb]{0,0,0}\makebox(0,0)[lb]{\smash{1}}}%
    \put(0.10726787,0.3574953){\color[rgb]{0,0,0}\makebox(0,0)[lb]{\smash{2}}}%
    \put(0.05094117,0.33456948){\color[rgb]{0,0,0}\makebox(0,0)[lb]{\smash{1}}}%
    \put(0.01734296,0.28104416){\color[rgb]{0,0,0}\makebox(0,0)[lb]{\smash{2}}}%
    \put(0.01205817,0.23680663){\color[rgb]{0,0,0}\makebox(0,0)[lb]{\smash{1}}}%
    \put(0.03993678,0.16204331){\color[rgb]{0,0,0}\makebox(0,0)[lb]{\smash{2}}}%
    \put(0.23503877,0.45491436){\color[rgb]{0,0,0}\makebox(0,0)[lb]{\smash{2}}}%
    \put(0.30418765,0.44412517){\color[rgb]{0,0,0}\makebox(0,0)[lb]{\smash{1}}}%
    \put(0.39060687,0.39297983){\color[rgb]{0,0,0}\makebox(0,0)[lb]{\smash{1}}}%
    \put(0.41879665,0.32442546){\color[rgb]{0,0,0}\makebox(0,0)[lb]{\smash{2}}}%
    \put(0.4354644,0.27351739){\color[rgb]{0,0,0}\makebox(0,0)[lb]{\smash{2}}}%
    \put(0.63338661,0.23874946){\color[rgb]{0,0,0}\makebox(0,0)[lb]{\smash{1}}}%
    \put(0.61957101,0.36210326){\color[rgb]{0,0,0}\makebox(0,0)[lb]{\smash{2}}}%
    \put(0.56388111,0.34299863){\color[rgb]{0,0,0}\makebox(0,0)[lb]{\smash{1}}}%
    \put(0.52773544,0.28119405){\color[rgb]{0,0,0}\makebox(0,0)[lb]{\smash{2}}}%
    \put(0.51671886,0.2248561){\color[rgb]{0,0,0}\makebox(0,0)[lb]{\smash{1}}}%
    \put(0.53631827,0.19658387){\color[rgb]{0,0,0}\makebox(0,0)[lb]{\smash{2}}}%
    \put(0.74288388,0.45506426){\color[rgb]{0,0,0}\makebox(0,0)[lb]{\smash{2}}}%
    \put(0.81458014,0.44427507){\color[rgb]{0,0,0}\makebox(0,0)[lb]{\smash{1}}}%
    \put(0.88826212,0.38039245){\color[rgb]{0,0,0}\makebox(0,0)[lb]{\smash{1}}}%
    \put(0.92982601,0.31884359){\color[rgb]{0,0,0}\makebox(0,0)[lb]{\smash{2}}}%
    \put(0.94585692,0.27366729){\color[rgb]{0,0,0}\makebox(0,0)[lb]{\smash{2}}}%
    \put(0.74236466,0.2189562){\color[rgb]{0,0,1}\makebox(0,0)[lb]{\smash{$[0,\infty)$}}}%
    \put(0.23134825,0.21858844){\color[rgb]{0,0,1}\makebox(0,0)[lb]{\smash{$[0,\infty)$}}}%
  \end{picture}%
\endgroup%

%% file: SmoothingRD.pdf_tex
%% Creator: Inkscape 0.48.5, www.inkscape.org
%% PDF/EPS/PS + LaTeX output extension by Johan Engelen, 2010
%% Accompanies image file 'SmoothingRD.pdf' (pdf, eps, ps)
%%
%% To include the image in your LaTeX document, write
%%   \input{<filename>.pdf_tex}
%%  instead of
%%   \includegraphics{<filename>.pdf}
%% To scale the image, write
%%   \def\svgwidth{<desired width>}
%%   \input{<filename>.pdf_tex}
%%  instead of
%%   \includegraphics[width=<desired width>]{<filename>.pdf}
%%
%% Images with a different path to the parent latex file can
%% be accessed with the `import' package (which may need to be
%% installed) using
%%   \usepackage{import}
%% in the preamble, and then including the image with
%%   \import{<path to file>}{<filename>.pdf_tex}
%% Alternatively, one can specify
%%   \graphicspath{{<path to file>/}}
%% 
%% For more information, please see info/svg-inkscape on CTAN:
%%   http://tug.ctan.org/tex-archive/info/svg-inkscape
%%
\begingroup%
  \makeatletter%
  \providecommand\color[2][]{%
    \errmessage{(Inkscape) Color is used for the text in Inkscape, but the package 'color.sty' is not loaded}%
    \renewcommand\color[2][]{}%
  }%
  \providecommand\transparent[1]{%
    \errmessage{(Inkscape) Transparency is used (non-zero) for the text in Inkscape, but the package 'transparent.sty' is not loaded}%
    \renewcommand\transparent[1]{}%
  }%
  \providecommand\rotatebox[2]{#2}%
  \ifx\svgwidth\undefined%
    \setlength{\unitlength}{211.44149997bp}%
    \ifx\svgscale\undefined%
      \relax%
    \else%
      \setlength{\unitlength}{\unitlength * \real{\svgscale}}%
    \fi%
  \else%
    \setlength{\unitlength}{\svgwidth}%
  \fi%
  \global\let\svgwidth\undefined%
  \global\let\svgscale\undefined%
  \makeatother%
  \begin{picture}(1,0.43070734)%
    \put(0,0){\includegraphics[width=\unitlength]{SmoothingRD.pdf}}%
    \put(0.0664453,0.04669733){\color[rgb]{0,0,0}\makebox(0,0)[lb]{\smash{$S^1$}}}%
    \put(0.79674747,0.04679599){\color[rgb]{0,0,0}\makebox(0,0)[lb]{\smash{$S^1$}}}%
  \end{picture}%
\endgroup%

%% file: uniformization.pdf_tex
%% Creator: Inkscape 0.48.5, www.inkscape.org
%% PDF/EPS/PS + LaTeX output extension by Johan Engelen, 2010
%% Accompanies image file 'uniformization.pdf' (pdf, eps, ps)
%%
%% To include the image in your LaTeX document, write
%%   \input{<filename>.pdf_tex}
%%  instead of
%%   \includegraphics{<filename>.pdf}
%% To scale the image, write
%%   \def\svgwidth{<desired width>}
%%   \input{<filename>.pdf_tex}
%%  instead of
%%   \includegraphics[width=<desired width>]{<filename>.pdf}
%%
%% Images with a different path to the parent latex file can
%% be accessed with the `import' package (which may need to be
%% installed) using
%%   \usepackage{import}
%% in the preamble, and then including the image with
%%   \import{<path to file>}{<filename>.pdf_tex}
%% Alternatively, one can specify
%%   \graphicspath{{<path to file>/}}
%% 
%% For more information, please see info/svg-inkscape on CTAN:
%%   http://tug.ctan.org/tex-archive/info/svg-inkscape
%%
\begingroup%
  \makeatletter%
  \providecommand\color[2][]{%
    \errmessage{(Inkscape) Color is used for the text in Inkscape, but the package 'color.sty' is not loaded}%
    \renewcommand\color[2][]{}%
  }%
  \providecommand\transparent[1]{%
    \errmessage{(Inkscape) Transparency is used (non-zero) for the text in Inkscape, but the package 'transparent.sty' is not loaded}%
    \renewcommand\transparent[1]{}%
  }%
  \providecommand\rotatebox[2]{#2}%
  \ifx\svgwidth\undefined%
    \setlength{\unitlength}{319.5805105bp}%
    \ifx\svgscale\undefined%
      \relax%
    \else%
      \setlength{\unitlength}{\unitlength * \real{\svgscale}}%
    \fi%
  \else%
    \setlength{\unitlength}{\svgwidth}%
  \fi%
  \global\let\svgwidth\undefined%
  \global\let\svgscale\undefined%
  \makeatother%
  \begin{picture}(1,0.26480025)%
    \put(0,0){\includegraphics[width=\unitlength]{uniformization.pdf}}%
    \put(0.69979689,0.14802965){\color[rgb]{0,0,0}\makebox(0,0)[lb]{\smash{\color{red}$S^1_{r_1}$}}}%
    \put(0.55650388,0.02322531){\color[rgb]{0,0,0}\makebox(0,0)[lb]{\smash{\color{red}$S^1_{r_2}$}}}%
    \put(0.35765786,0.12961474){\color[rgb]{0,0,0}\makebox(0,0)[lb]{\smash{\color{red}$\gamma_1$}}}%
    \put(0.398691,0.22706392){\color[rgb]{0,0,0}\makebox(0,0)[lb]{\smash{\color{red}$\gamma_2$}}}%
  \end{picture}%
\endgroup%